\documentclass[reqno,12pt]{amsart}
\date{23 December 2025}

\usepackage[english]{babel}
\usepackage{fullpage}
\usepackage{enumerate}

\usepackage[centertags]{amsmath}
\usepackage{indentfirst}
\usepackage{fancyhdr}
\usepackage[dvips]{graphicx}
\usepackage{psfrag}
\usepackage{stmaryrd}
\numberwithin{equation}{section}
\usepackage[latin1]{inputenc}
\usepackage{enumitem}

\usepackage[hidelinks,pdfstartview=FitH]{hyperref} 
\usepackage{amssymb,interval,mathrsfs,enumitem}
\usepackage{tikz}
\usepackage[T1]{fontenc}
\setenumerate{label=\textup{(\roman*)},itemsep=2pt,topsep=2pt,leftmargin=2.5em}
\usetikzlibrary{arrows,cd,arrows.meta,calc}

\usepackage{mathtools}
\mathtoolsset{showonlyrefs}

\let\oldtocsection=\tocsection
\let\oldtocsubsection=\tocsubsection
\renewcommand{\tocsection}[2]{\hspace{0em}\oldtocsection{#1}{#2}}
\renewcommand{\tocsubsection}[2]{\hspace{1em}\oldtocsubsection{#1}{#2}}

\allowdisplaybreaks[4]
\numberwithin{equation}{section}

\newtheorem{prop}{Proposition}[section]
\newtheorem{lem}[prop]{Lemma}
\newtheorem{cor}[prop]{Corollary}

\newtheorem{thm}[prop]{Theorem}

\theoremstyle{remark}
\newtheorem{rem}[prop]{Remark}
\newtheorem{ex}[prop]{Example}

\usepackage{etoolbox}
\AtEndEnvironment{ex}{\hfill$\blacksquare$}

\newcommand{\ot}{\otimes}
\newcommand{\otb}{\otimes_B}
\newcommand{\beq}{\begin{equation}}
\newcommand{\eeq}{\end{equation}}
\usepackage{etoolbox}
\AtEndEnvironment{rem}{\hfill\ensuremath{\square}}
\AtEndEnvironment{oss}{\hfill\ensuremath{\square}}

\newcommand{\nn}{\nonumber}

\newcommand{\N}{\mathbb{N}}
\newcommand{\Z}{\mathbb{Z}}

\newcommand{\inner}[2]{\left<#1,#2\right>}

\renewcommand{\tt}{\mathsf{t}}
\renewcommand{\ss}{\mathsf{s}}

\newcommand{\id}{\mathrm{id}}

\newcommand{\kk}{\textbf{k}}
\newcommand{\inc}{u}
\newcommand{\qbin}[2]{{\genfrac{[}{]}{0pt}{}{#1}{#2}}}
\newcommand{\ta}{\psi}
\newcommand{\tap}{\overline{\ta}}
\newcommand{\tain}{\psi^{-1}}
\newcommand{\dta}{\cdot^\ta}
\newcommand{\dtain}{\cdot^{\tain}}
\newcommand{\taa}[1]{{#1}^{[\ta]}}
\newcommand{\taain}[1]{{#1}^{[\tain]}}

\newcommand{\taadue}[1]{{#1}^{[\ta]_{\scriptscriptstyle{2}}}}
\newcommand{\taatre}[1]{{#1}^{[\ta]_{\scriptscriptstyle{3}}}}
\newcommand{\taaqu}[1]{{#1}^{[\ta]_{\scriptscriptstyle{4}}}}
\newcommand{\taacq}[1]{{#1}^{[\ta]_{\scriptscriptstyle{5}}}}

\newcommand{\hg}{H_\cot}

\renewcommand{\cot}{\gamma}
\newcommand{\co}[2]{\cot\left({#1}\ot{#2}\right)}
\newcommand{\coin}[2]{\bar\cot\left({#1}\ot{#2}\right)}
\newcommand{\mt}{\cdot_\cot}
\newcommand{\mtco}{\bullet_\cot}

\newcommand{\coL }{\mathcal{L}}
\newcommand{\coM }{\mathcal{M}}

\newcommand{\pf}{C \ot^\ta_B A}
\newcommand{\tb}{\widetilde{B}}
\newcommand{\ff}{\mathsf{F}}

\newcommand{\M}{\mathrm{M}}

\newcommand{\zero}[1]{{#1}_{\scriptscriptstyle{(0)}}}
\newcommand{\one}[1]{{#1}_{\scriptscriptstyle{(1)}}}
\newcommand{\two}[1]{{#1}_{\scriptscriptstyle{(2)}}}
\newcommand{\three}[1]{{#1}_{\scriptscriptstyle{(3)}}}
\newcommand{\four}[1]{{#1}_{\scriptscriptstyle{(4)}}}
\newcommand{\tone}[1]{{#1}^{\scriptscriptstyle{<1>}}}
\newcommand{\ttwo}[1]{{#1}^{\scriptscriptstyle{<2>}}}
\newcommand{\lone}[1]{{#1}^{\scriptscriptstyle{<1>}}}
\newcommand{\ltwo}[1]{{#1}^{\scriptscriptstyle{<2>}}}

\newcommand{\coone}[1]{{#1}_{\scriptscriptstyle{\underline{(1)}}}}
\newcommand{\cotwo}[1]{{#1}_{\scriptscriptstyle{\underline{(2)}}}}

\title{Push-forward of Hopf--Galois extensions: \\the non central case}
\author{Giovanni Landi, Chiara Pagani}
\address[G.~Landi]
{Universit\`a di Trieste and INFN--Trieste, Trieste, Italy}
\email{landi@units.it}

\address[C.~Pagani]
{Universit\`a di Napoli Federico II and INFN--Napoli, Napoli, Italy}
\email{chiara.pagani@unina.it}

\begin{document}

\maketitle

\begin{abstract}
We study the push-forward of Hopf--Galois extensions as the algebraic counterpart of the pullback of principal bundles. 
We apply the theory of twisted tensor product algebras to endow covariant extensions of modules along a map $\mathsf{F}$ with an algebra structure, under compatibility conditions between $\mathsf{F}$ and the twisting map. 
The push-forward of an $H$-Galois extension $B \subset A$ along a map $\mathsf{F} : B \to C$ is an $H$-Galois extension of $C$. The corresponding Ehresmann--Schauenburg algebroids are compared. 
\end{abstract}

\tableofcontents
\parskip = .75ex

\section{Introduction}
From an algebraic viewpoint, particularly within algebraic and noncommutative geometry, principal bundles are modeled by faithfully flat Hopf--Galois extensions \cite{Sch90}. In this context, the theory of pull-back or induced bundles has not yet reached a satisfactory formulation. Given an $H$-Galois algebra extension $B \subset A$ and an algebra map $\ff:B \to C$,  
the covariant $\ff$-extension $C \otb A$ of  the $B$-module $A$  along $\ff$ (see \eqref{push}) can be viewed, as a linear space, as the analogue of the fibered product of spaces along a map -- that is, the total space of a pull-back bundle.  The usual multiplication
$m_{C \ot A} = (m_C \ot m_A) \circ (\id_C \ot \textup{flip} \ot \id_A)$ of the tensor product algebra  $C \ot A$ descends to a well-defined multiplication on the 
covariant $\ff$-extension $C \ot_B A$   under the assumption that the algebra $B$ is central in $A$, and hence in particular commutative, and that the image $\ff(B)$ of $\ff$  is central in $C$, \cite{kassel}.  The problem of endowing $C \ot_B A$ with an algebra structure outside the setting of commutative algebra extensions remains open.
In this work we address this problem by employing the theory of twisted tensor product algebras.
   
Twisted tensor product algebras, in the setting here considered, were first studied in \cite{csv} (where differential calculi on them were also investigated \cite[\S 3]{csv}) and later dualized to the context of coalgebras and bialgebras in \cite{cae}. Given two associative algebras $A$ and $C$, the vector space $C \ot A$ is endowed with the multiplication $m^\ta:=(m_C \ot m_A) \circ (\id_C \ot \ta \ot \id_A)$, where  $\ta : A \ot C \to C \ot A$ is any linear map satisfying the minimal conditions required to ensure the associativity of the twisted multiplication $m^\ta$ (see \S  \ref{sec:twist}).
The idea of creating a new object by deforming the tensor product of two existing ones predates the aforementioned work and has been explored in various contexts. Notable examples in Hopf algebra theory are  Takeuchi's smash product  and crossed products (see e.g. \cite{mon}), Drinfel'd quantum doubles and Majid's  bicrossproducts of Hopf algebras  (see e.g. \cite{majid}). Twisted tensor products of algebras, together with their counterparts for coalgebras and Hopf algebras, admit a natural interpretation within Beck's theory of distributive laws (cf. \eqref{C1} and \eqref{C2}), which offers the categorical setting for these constructions.

Here we apply the theory of twisted tensor product of algebras to the study of push-forward of possibly non central Hopf--Galois extensions. 
Given two unital associative algebras $C$ and $A$ and an algebra map $\ff : B \to C$, for $B$ a subalgebra of $A$, 
we endow $C \ot A$ with a twisted  algebra multiplication $m^\ta$. We study  compatibility conditions between the twisting map  $\ta$ and the map $\ff : B \to C$ which ensure that the algebra structure of $C \ot^\ta \!A =(C \ot A, m^\ta)$ descends to the quotient $C \ot_B A$ (Proposition \ref{prop:Blin}). The resulting algebra is referred to as the push-forward algebra. Under assumptions on $\ta$, it inherits an $H$-comodule algebra structure from the one on $A$ (Proposition \ref{prop:Hcas}).  As a generalisation of \cite[Thm.~9]{Ru98} and \cite[Thm.1.4]{kassel}   to possibly non central $H$-Galois extensions, we prove  the following result:
 
\noindent
\textbf{Theorem \ref{thm:pushHG}.}
\textit{Let $B \subset A$ be an $H$-Galois extension with $A$ faithfully flat as a $B$-module. Let  
$\ta:A\ot C\to C\ot A$ be a twisting map with properties as in Lemma \ref{lem:CAB}. Then 
the twisted push-forward algebra $\pf$ of $B \subset A$ along the map $\ff:B \to C$ is a faithfully flat $H$-Galois extension of $C$.
}

The structure of the paper is as follows. Section \ref{sec:twisted-tp} reviews some preliminary results on twisted tensor products of algebras  and presents some basic facts  on the invertibility of twisting maps and their compatibility with $*$-structures. The main section, Section   \ref{sec:twistedHG},   applies the  theory  of twisted tensor algebras to the study of an algebra structure for covariant $\ff$-extensions of a $B$-module $A$ (in \S\ref{sec:pft}). The construction is specialized to the case where $B \subset A$ is an Hopf--Galois extension in \S \ref{sec:pf-HG}.  
Section \ref{sec:cases} presents some particular cases of push-forward of Hopf--Galois extensions $B \subset A$: the case of Galois objects (\S \ref{sec:Gobjects}) and those of push-forwards along the (injective) algebra inclusion $i: B \to A$ (\S \ref{sec:pf-tot}) and along a (surjective) projection map $\pi: B \to B/I$ to a quotient algebra (\S \ref{sec:quotient}). 
In Section \ref{sec:ES} we consider the Ehresmann--Schauenburg bialgebroid of a Hopf--Galois extension obtained via the push-forward construction and compare it with that of the original 
Hopf--Galois extension.
Finally Appendix \ref{appA} reviews some aspects of extensions of modules and Appendix \ref{app:B} shows that a property of a strong connection is preserved under $2$-cocycle deformation.

\textit{Notation.}  We work over a field $\kk$.  
Given an algebra $A$, we denote by $m_A$ and $\eta_A$ the multiplication and the unit map respectively, and use the  notation $1_A=\eta_A(1_\kk)$ for the unit element of $A$. 
For a Hopf algebra $H$, we denote $\Delta$, $\varepsilon$ and $S$ its coproduct, counit and antipode respectively. We use the standard Sweedler's notation for the coproduct and for the coaction of an Hopf algebra.

\section{Twisted tensor products}\label{sec:twisted-tp}

In this introductory section, we recall in \S\ref{sec:twist}  some basic facts about twisted tensor products of algebras (see e.g. \cite{csv,cae}) and prove in \S \ref{sec:inv*} some  results on the invertibility of twisting maps and their compatibility with $*$-structures.

 \subsection{Twisted tensor products of algebras}\label{sec:twist}
 
Tensor product algebras can be generalised to twisted tensor products by replacing the flip map by a more general suitable map.

Let $A$ and $C$ be two unital and associative algebras and $\ta:A\ot C\to C\ot A$ a linear map. The map 
\beq \label{opsi}
m^\ta:=(m_C\ot m_A)(\id_C \ot \ta \ot \id_A):(C \ot A)\ot ( C \ot A)\to C \ot A 
\eeq
defines an associative product on the vector space $C \ot A$ if and only if the map $\ta$ satisfies
\begin{multline}
\label{Oeq}
(\id_C \ot m_A)\circ(\ta\ot \id_A)\circ(\id_A\ot m_C\ot \id_A)\circ(\id_A\ot \id_C \ot \ta) = \\ (m_C \ot \id_A )\circ(\id_C \ot \ta)\circ(\id_C \ot m_A \ot \id_C)\circ(\ta \ot \id_A \ot \id_C) \, .
\end{multline}

We refer to a  linear map
$\ta:A \ot C \to C\ot A$
satisfying condition \eqref{Oeq} as a  \textit{twisting map}, and write  
 $\ta(a \ot c):= \taa{c} \ot \taa{a}$ for the image of $\ta$ (with possible sum understood).
Recursively, we use the notation $[\ta]_2, \dots, $ for the image of $\ta$ when applied more than once.
Condition \eqref{Oeq} then reads
\beq\label{Oeq-el}
\taadue{(c \taa{c'})} \ot \taadue{a} \taa{a'}= \taa{c} \taadue{c'} \ot \taadue{(\taa{a}a')} \, , \quad
\forall \, a,a' \in A, \; c,c' \in C \, .
\eeq

We denote $C \ot^\ta \!A$ the vector space $C\ot A$ equipped with the twisted product $m^\ta$ and refer to it as the \textit{twisted tensor product algebra}, $C \ot^\ta \!A= (C \ot A, m_{\ta})$. Also, we use the symbol $\dta$ to indicate the product of two elements: 
\beq\label{mult-psi}
m^\ta \Big( (c \ot a) \ot (c' \ot a') \Big) = (c \ot a) \dta (c' \ot a')  = 
c \taa{c'} \ot \taa{a} a'.
\eeq
The terminology "twisted tensor product" is taken from \cite{VD-VK}, whereas \cite{cae} uses "smash product", being $m^\ta$ a generalization of Takeuchi's smash product.

A sufficient condition for the associativity  \eqref{Oeq} to hold is that the following two conditions (distributive laws)  are satisfied:
\begin{align}
\label{C1}
\ta (m_A \ot \id_C )&=(\id_C \ot m_A)\circ (\ta \ot \id_A)\circ (\id_A \ot \ta) 
\\
\label{C2}
\ta (\id_A \ot m_C )&= (m_C \ot \id_A)\circ (\id_C \ot \ta)\circ (\ta \ot \id_C) \, .
\end{align}
They read, respectively,
 \begin{align}
 \label{C1el}
\ta(aa' \ot c) &=  \ta \left( a \ot {\taa{c}} \right)  \taa{a'} =  \taadue{\taa{c}} \ot \taadue{a}\taa{a'} 
\\
\label{C2el}
\ta(a \ot cc') &= \taa{c}  \ta(\taa{a} \ot c')=\taa{c}\taadue{c'}  \ot \taadue{\taa{a}} 
\end{align}
 on elements $a,a' \in A$, $c,c' \in C$. 

Given a twisting map $\ta:A \ot C \to C\ot A$,   
the tensor unit $\eta_\ot:=\eta_C \ot \eta_A$ is compatible with the twisted product $m^\ta$, that is $m^\ta \circ (\id \ot \eta_\ot) = \id =m^\ta \circ(\eta_\ot \ot \id)$  if and only if $\ta$ is normal, that is
for all $c \in C$, $a \in A$,
\begin{equation}\label{nor}
\ta  (1_A \ot c)= c \ot 1_A \quad \mbox{(right normal)}, \qquad  \ta (a \ot 1_C)=1_C \ot a\quad \mbox{(left normal)} .
\end{equation}

The associativity of $m_{\ta}$ and the property of $\eta_{ \ot }$ to be a compatible unit are independent: there are examples of associative algebras $C \ot^\ta \!A$ which do not admit a unit  (and,  clearly, not every  normal linear map $\ta:A \ot C \to C\ot A$  satisfies the associativity condition \eqref{Oeq}). 

The condition \eqref{Oeq} and left, respectively right, normality of $\ta$
 imply the  condition \eqref{C1}, respectively  \eqref{C2}. Thus, for a normal twisting map,   \eqref{C1}  and  \eqref{C2} are equivalent to the associativity of the product $m^\ta$.
 If a twisting map is right normal and satisfies the distributive law \eqref{C1} then it also satisfies
 the distributive law \eqref{C2}. Similarly a left normal twisting map that  satisfies \eqref{C2} also satisfies \eqref{C1}.
 
For finitely generated algebras $A$ and $C$, equations \eqref{C1} and \eqref{C2} imply that the linear map $\ta$ is  determined by its value on the algebra generators of $A$ and $C$ since there is a unique way to extend it to the whole of $A \ot C$ once is (well-)defined on the algebra generators.

Twisted tensor product of algebras are closely connected to the factorization problem for an algebra (see \cite[Ch.~7]{majid}, \cite[Thm.~2.10]{cae}). 

\begin{thm}\label{thm-cae-alg}

Let $C, A$ and $X$ be algebras. The following two conditions are equivalent:
\begin{enumerate}
\item
There exists an algebra isomorphism $X \simeq C \ot^\ta \!A$ for some normal twisting map  $\ta$.
\item The algebra X factorizes through $C$ and $A$, that is  there are two algebra morphisms $u_C: C \to X$ and $u_A : A \to X $ such that the map $\xi:= m_X \circ (u_C \ot u_A): C\ot A \to X$ is an isomorphism of vector spaces.
\end{enumerate}
\end{thm}
The proof is based on the fact that if $C \ot^\ta \!A$ is a twisted tensor product algebra coming from a normal twisting map  $\ta$, then the maps  \begin{equation}
\label{projpi_alg} 
u_C  : C \to C \ot^\ta \!A \, , \quad c \mapsto c \ot 1_A \, , \qquad
u_A  : A \to C \ot^\ta \!A \, , \quad a \mapsto 1_C \ot a \, ,
\end{equation}    
are algebra morphisms, and the map 
$
\xi:=m^\ta \circ (u_C \ot u_A):C \ot A   \to C \ot^\ta \!A
$
is an isomorphism of vector spaces, being  $\xi(c \ot a)= c \ot a$. For the opposite implication, one constructs the normal twisting map as $\ta=\xi^{-1}\circ m_X \circ  (u_A \ot u_C)$.

The  construction of twisted tensor products was generalised to coalgebras and bialgebras in \cite[\S 2]{cae}.
The mixed case of $C$ a coalgebra and $A$ an algebra was studied in \cite[\S 2]{BrMa98} where entwining structures were used to  generalise Hopf--Galois extensions to coalgebra principal bundles (see also \cite[\S 34]{BW}).

For examples of twisted tensor product of algebras see e.g. \cite[\S 2]{cae}. 
Here we recall only the fundamental example of Majid's double crossed products (and, in particular, smash products), of which twisted tensor products are a generalization.   
\begin{ex}\label{ex:smash} The basic example of a twisting map is the one arising from the double crossed product (or bicrossproduct) construction of bialgebras. Let us briefly recall this framework, see e.g. \cite[\S 10.2.5]{KS}. 
Let $\{C, A\}$ be a matched pair of bialgebras: $(A, \triangleleft)$ is a right $C$-module coalgebra and $(C, \triangleright)$ is a left $A$-module coalgebra with the following compatibility conditions of the two actions
\begin{align}
&(aa') \triangleleft c = (a \triangleleft (\one{a'} \triangleright \one{c})) (\two{a'} \triangleleft \two{c}) \, ,  \qquad 1_A \triangleleft c = \varepsilon_C(c)
\\
&a \triangleright (cc') = (\one{a} \triangleright \one{c}) ((\two{a} \triangleleft \two{c}) \triangleright c')\, ,   \qquad a \triangleright 1_C = \varepsilon_A(a)
\\
&(\one{a} \triangleleft \one{c}) \ot (\two{a} \triangleright \two{c}) = (\two{a} \triangleleft \two{c}) \ot (\one{a} \triangleright \one{c}) .
\end{align}
The  vector space $C \otimes A$ is an algebra  with product 
\beq
(c \ot a) \cdot^\ta (c' \ot a') =  c (\one{a} \triangleright \one{c'}) \ot (\two{a} \triangleleft \two{c'}) a',
\eeq
corresponding to the twisting map
\beq\label{ta-double}
\ta: A \ot C \to C \ot A : \, a \ot c \mapsto (\one{a} \triangleright \one{c}) \ot (\two{a} \triangleleft \two{c}).
\eeq
The resulting  algebra is a bialgebra with the usual tensor product coalgebra structure of $C \ot A$. It is called the double crossed product bialgebra of the pair $\{C,A\}$ and denoted $C \bowtie A$.  When one of the two actions is trivial the construction reduces to smash (or crossed) products.
\end{ex}

\subsection{Invertible twisting maps and $*$-structures}\label{sec:inv*}
We next comment on the existence of a  $*$-structure for a twisted tensor product algebra and its relation to the invertibility of the  twisting map.

Suppose  $\ta:A\ot C\to C\ot A$  is a  twisting map which is invertible as a linear map. Let
$\ta^{-1}: C\ot A \to A\ot C$ denote its inverse.  
Associativity of the product defined by an invertible map $\ta$ does not in general imply associativity for the product defined by $\ta^{-1}$; this implication requires the stronger  distributive laws \eqref{C1} and \eqref{C2}.
If $\ta:A\ot C\to C\ot A$  is a linear map that satisfies the conditions \eqref{C1} and \eqref{C2}, so does its inverse and thus $\ta^{-1}$ defines an
associative product $m^{\tain}$ for $A \ot C$.  The resulting
twisted tensor product algebra $A\ot^{\tain} \!\!C$ is isomorphic to $C \ot^\ta \!A$.

\begin{lem}
Let $\ta:A\ot C\to C\ot A$  be a linear map that satisfies the conditions \eqref{C1} and \eqref{C2}.
The twisted tensor product algebras $A\ot^{\tain} \!\!C$ and $C \ot^\ta  A$ are isomorphic via the twisting map, that is 
$$
\ta: (A\ot^{\tain} \!\!C, \, m^{\tain})  \longrightarrow (C \ot^\ta  A, \, m^\ta)
$$
is an algebra isomorphism:
\beq\label{twist-am} 
\ta \big( (a \ot c) \dtain (a'  \ot  c' )\big)  =  \big( \ta(c \ot a) \big) \dta \big(\ta (c'  \ot  a' )\big).
\eeq
\end{lem}
\begin{proof}
The combination of \eqref{C1el} and \eqref{C2el} gives
\beq\label{C12}
\ta(aa' \ot cc') =  \taatre{\taa{c}} \taaqu{\taadue{c'}}   \ot \taaqu{\taatre{a}} \taadue{\taa{a'}}  .
\eeq
Using this we have
\begin{align*}
\ta(a \taain{a'} \ot \taain{c} c')  
&=   \taatre{\taa{\taain{c} }} \taaqu{\taadue{c'}}   \ot \taaqu{\taatre{a}} \taadue{\taa{\taain{a'}}} 
=   \taadue{c}   \taatre{\taa{c'}}   \ot \taatre{\taadue{a}}  \taa{a'}
\\
&=   \taadue{c}  \ta \big(\taadue{a} \ot  \taa{c'}  \big)  \taa{a'}
=  (\taadue{c}  \ot \taadue{a}) \dta (  \taa{c'} \ot \taa{a'})
\end{align*}
thus showing \eqref{twist-am}.
\end{proof}
Assuming now that the algebras $A$ and $C$ are $*$-algebras there are two natural $*$-structures on the twisted tensor product algebra that can be considered.
\begin{lem} 
Let $A$ and $C$ be $*$-algebras, with involutions denoted  $*_A$ and $*_C$ respectively. 
Let $\ta:A\ot C\to C\ot A$  be a twisting map, with $\ta(1_A \ot 1_C)= 1_C \ot 1_A$.
Then the  twisted tensor product algebra $C \ot^\ta \!A$  is a $*$-algebra with involution $*_C \otimes *_A$ if and only if $\ta$ is the flip map, $\ta(c \ot a) = a \ot c$.
\end{lem}
\begin{proof}
One implication is direct. For the other implication, suppose $C \ot^\ta \!A$  is a $*$-algebra with involution $*_C \otimes *_A$, that is
\begin{align*} 
\big((c \ot a) \dta (c' \ot a')\big)^*  
&= (c \taa{c'})^* \ot (\taa{a} a')^*
= (\taa{c'})^* c^* \ot a'^* (\taa{a})^* 
\\
&
= (c'^* \ot a'^*) \dta   (c^* \ot a^*)  = 
c'^* \taa{(c^*)} \ot \taa{(a'^*)} a^*
\end{align*}
for all $a,a' \in A$ and $c,c' \in C$.
Then, in particular, for $a' \ot c=1_A \ot 1_C$, the identity
 gives
$\taa{c'} \ot \taa{a} = c'   \ot  a$ for all $a \in A$, $c' \in C$, that is $\ta$ is the flip map.
\end{proof}
\begin{lem}\label{lem:*inv}
Let $\ta:A\ot C\to C\ot A$  be a linear map.
The map 
\beq\label{*ta}
* := \ta (*_A \otimes *_C)\circ \textup{flip}: C \ot A \to C \ot A \, , \quad  c \ot a \mapsto (c \ot a)^*:=\ta(a^* \ot c^*)
\eeq
is an involution
 if and only if
$\ta$ is invertible with inverse 
\beq\label{ns*ta}
\ta^{-1} =  (*_A \otimes *_C) \circ \textup{flip} \circ \ta \circ (*_A \otimes *_C)\circ \textup{flip} : C \ot A \to A \ot C.
\eeq
In this case, for $\ta$ satisfying  \eqref{C1} and \eqref{C2}, the involution in  \eqref{*ta} defines a $*$-structure on the twisted tensor product algebra $C \ot^\ta \!A$.
\end{lem}
\begin{proof}
The  condition for  \eqref{*ta} to be an involution, 
\beq\label{invol}
\left( \ta (a^*  \ot c^*) \right)^* = {\taadue{{\taa{c}}^*}}^* \ot {\taadue{{\taa{a}}^*}}^*= c \ot a \; , \quad  \forall \, a \in A ,\; c\in C ,
\eeq
is equivalent to the   map $\ta^{-1} $ in \eqref{ns*ta} being a left and right inverse of $\ta$.

As for the second statement, we have to show that \eqref{*ta} is an anti-algebra map for the product $m_\ta$ of $C \ot^\ta \!A$.  
For $a,a' \in A$ and $c,c' \in C$:
\begin{align*}
\big((c \ot a) \dta (c' \ot a')\big)^*  
&
= \big( c \taa{c'} \ot \taa{a} a'\big)^*  
= \ta  \big((\taa{a} a')^*  \ot (c \taa{c'})^* \big)
= \ta  \big(a'^*(\taa{a})^*  \ot ( \taa{c'})^* c^* \big)
\\
&= 
 \taaqu{\taadue{( \taa{c'})^*}} \taacq{\taatre{c^* }}   \ot \taacq{\taaqu{(a'^*)}}\taatre{\taadue{   (\taa{a})^*   }}  
\end{align*}
where we used 
the combination \eqref{C12} of \eqref{C1el} and \eqref{C2el}
for the last equality.
By using \eqref{invol}, the expression simplifies to  
\begin{align*}
\big((c \ot a) \dta (c' \ot a')\big)^*  
&=  \taadue{(c'^*)} \taatre{\taa{(c^* )}}   \ot \taatre{\taadue{(a'^*)}} \taa{(a^*)} 
\\
&=  \big( \taadue{c'^*}  \ot \taadue{a'^*}  \big)\dta \big( \taa{c^*}  \ot \taa{a^*} \big)
=(c' \ot a')^* \dta (c \ot a)^* .\qedhere
\end{align*} 
\end{proof}
In the hypothesis of the previous Lemma, being $\ta$ invertible, one can also form the twisted tensor product algebra $A\ot^{\tain} \!\!C$ 
with associative product $m^{\ta^{-1}}$. Condition \eqref{ns*ta} is symmetric for the exchange of $\ta$ with $\ta^{-1}$ (and $A \leftrightarrow C$), being  $(*_A \otimes *_C) \circ \textup{flip}  =    \textup{flip}\circ (*_C \otimes *_A)$, and thus $A\ot^{\tain} \!\!C$ becomes
 a $*$-algebra with 
$$(a \ot c)^*:= \tain(c^* \ot a^*) = (\taa{a})^* \ot (\taa{c})^* .$$ The map 
 $\ta: A\ot^{\tain} \!\!C \to C \ot^\ta  A $ is a $*$-algebra isomorphism.

\section{Push-forward of Hopf--Galois extensions}\label{sec:twistedHG} 

 This section applies the theory of twisted tensor product algebras to the study of push-forwards of Hopf--Galois extensions.
In \S \ref{sec:pft} 
we study conditions for which the multiplication $m^\ta$ of a twisted tensor product algebra 
$C \ot^\ta \!A$ descends to a well-defined algebra structure on  the  covariant $\ff$-extension $C \ot_B A$ of $A$ via a map  $\ff: B \to C$, for $B$ a subalgebra of $A$. 
Then, in \S \ref{sec:pf-HG}, we specify the construction to the case of $B \subset A$ being an $H$-Galois extension. In this context,  $C \ot_B A$ is seen as an algebraic counterpart of a fibered product of the total space of a principal bundle along a map on the base space.   
We thus study  push-forwards of $H$-Galois algebra extensions as the algebraic analogues of  pull-backs of principal bundles.

\subsection{Twisted products for covariant extensions}\label{sec:pft}

We determine the conditions under which the multiplication $m^\ta$ of a twisted tensor product algebra 
$C \ot^\ta \!A$ descends to the quotient of $C \ot A$ over a subalgebra $B \subset A$. 

Let $A$ and $C$ be  two algebras and 
let $B$ be a subalgebra of $A$ with $\ff: B \to C$ an algebra morphism. 
The quotient $C \ot_B A $, also called the push-forward of $A$ along $\ff$, is defined  as the covariant $\ff$-extension of the $B$-module $A$ via $\ff$: 
\beq\label{push}
C \ot_B A := C \ot A /{\langle c \ot b a - c \, \ff(b) \ot a  \; , \; a \in A, \,b \in B, \, c \in C \rangle} \, .
\eeq
It  is a left $C$-module  and a right $A$-module  with commuting actions $c'(c \ot a)= c'c \ot a$ and $(c \ot a)a'= c \ot aa'$, for $a,a' \in A$, $c, c' \in C$. We denote $\pi_B: C \ot A \to C \ot B$, $\pi_B(c \ot a):= c \otb a $, the quotient map; by construction, it is a morphism of left $C$-modules and right $A$-modules.
We also recall that 
if $A$  is faithfully flat as a left $B$-module, then  $C \ot_B A$  is faithfully flat as a left $C$-module, \cite[Prop. 5 Ch.I \S 3]{bour}.

The problem of giving $C \ot_B A$ an algebra structure is open.   
The usual multiplication
$m_{C \ot A} = (m_C \ot m_A) \circ (\id_C \ot \textup{flip} \ot \id_A)$ of the tensor product algebra  $C \ot A$ descends to a well-defined multiplication on the 
push-forward $C \ot_B A$   under the hypothesis that the algebra $B$ is central in $A$, and thus in particular commutative, and that the image $\ff(B)$ of $\ff$  is central in $C$, \cite{kassel}.  

Here we avoid the centrality condition by giving $C \ot A$   a twisted  algebra multiplication $m^\ta$, and we study  compatibility conditions between the twisting map  $\ta$ and the map $\ff$ which ensure that the algebra structure of $C \ot^\ta \!A =(C \ot A, m^\ta)$ descends to the quotient $C \ot_B A$.

Consider  the vector space $C$ as a $B$-bimodule via the map $\ff$   and   $A$,  $A \ot C$, $C \ot A$ and $C \otb A$ as $B$-bimodules in the obvious way.
We write   $\tap:= \pi_B \circ \ta: A \ot C \to C \otb A$ for the composition of  $\ta$ with the projection map $\pi_B$. 
We have the following result.

\begin{prop}\label{prop:Blin}
Let $\ta:A\ot C\to C\ot A$  be a twisting map  and $C \ot^\ta \!A = (C \ot A, m^\ta)$ be the corresponding twisted tensor product algebra. 
The  multiplication $m^\ta$, defined as in \eqref{opsi}, descends to a well-defined multiplication on 
the push-forward
 $C \ot_B A$  if and only if the map $\tap$ is a $B$-bimodule morphism, that is,  
\beq\label{f-lin}
\tap (b (a \ot c))= \tap (ba \ot c) = \ff(b) \tap (a \ot c) \, \,  ; \quad \tap ((a \ot c)b)= \tap (a \ot c \, \ff(b)) = \tap (a \ot c) b 
\eeq
for all $a \in A$, $b \in B$, $c \in C$.
\end{prop}
\begin{proof}
The map $m^\ta$ descends to a well-defined map on 
 $(C \ot_B A) \ot (C \ot_B A)$  if and only if 
$$
 \pi_B \Big( \ta (b a \ot c)b'-  \ff(b) \ta (a \ot c \, F(b')) \Big)=0 \quad ; \quad
 \pi_B \Big( \ta (b a \ot c \, \ff(b')) -  \ff(b) \ta (a \ot c ) b \Big)=0
 $$
 for all $a \in A$, $b,b' \in B$ and $c  \in C$. That is,  being $\pi_B$ a $B$-bimodule map,
 if and only if 
 $$
   \tap (b a \ot c)b'= \ff(b) \tap (a \ot c \, \ff(b'))  \quad ; \quad
   \tap (b a \ot c \, \ff(b') ) =  \ff(b) \tap (a \ot c ) b.
 $$
 It is then straightforward to check that these identities are equivalent to those in \eqref{f-lin}.
 \end{proof}
The resulting algebra
 $C \ot_B A$,  with product $m^\ta$,  will be denoted by $\pf$ and  called the 
 twisted push-forward algebra. 
 \begin{rem}
The requirement that $m^\ta$ descends to a well-defined map on the quotient $C \otb A$ is equivalent 
 to that of  $I:={\langle c \ot b a - c \, \ff(b) \ot a  \; , \; a \in A, \,b \in B, \, c \in C \rangle}$ being an algebra ideal:
 \begin{align*}
 (c' \ot a') \cdot^\ta (c \ot b a - c \, \ff(b) \ot a) \in I \quad \iff \quad \pi_B \Big( \ta (a' \ot c)b -   \ta (a' \ot c \, F(b)) \Big) =0
 \\
 (c \ot b a - c \, \ff(b) \ot a) \cdot^\ta (c' \ot a') \in I \quad \iff \quad \pi_B \Big( \ta (ba \ot c') -   \ff(b) \ta ( a \ot c') \Big) =0 \, ,
 \end{align*}
 for all $a,a' \in A$, $b \in B$ and $c,c'  \in C$.
We note that the $B$-linearity condition \eqref{f-lin} of the map $\tap$ is weaker  than an analogous $B$-linearity condition for the map $\ta$, $\ta(ab \ot c \ff(b'))= \ff(b) \ta (a \ot c) b$.  The $B$-linearity of $\ta$ would corresponds to the requirement that the product of an element of the algebra $C \ot^\ta \!A$ and one of $I$  is zero. 
\end{rem}
\begin{rem} The previous remark also indicates that for a surjective map
$\ff$ it would be too strong to assume $B$-linearity of $\ta$ . 
For  $\ff=B \to C$  surjective, the  $B$-linearity  of a twisting map $\ta:A\ot C\to C\ot A$  would require
$\ta (a \ot c)=   \ta (a \ot 1_C) b_c$ for all $a \in A$, $c \in C$ and  $b_c \in B$ any element such that $\ff(b_c)=c$.  
Thus, in order for this map to be well-defined, one would also need
$1_C \ot a \, b_c = 1_C \ot a \, b'_c$  for any two elements $b_c, \,b'_c \in \ff^{-1}(c)$.
Moreover $\ta$ could not be (right) normal because this would require  $c \ot 1_A = \ta(1_A \ot c) = 1_C \ot b_c$ for all $c \in C$ and $b_c \in \ff^{-1}(c)$.
\end{rem}
 
We conclude this subsection with a comment on invertible twisting maps and $*$-structures.
Suppose the twisting map $\ta$ is invertible with inverse map $\ta^{-1}$ that is a twisting map as well (see \S \ref{sec:inv*}).  From  Proposition \ref{prop:Blin}, the multiplication $m^{\ta^{-1}}$ of  $A \ot^{\ta^{-1}} \!\!C$ descends to a well-defined algebra structure on 
the push-forward $A \ot_B C := A \ot C /J$ if and only if   $J:={\big\langle a \, b \ot c - a \ot  \ff(b) \, c  \; , \; a \in A, \,b \in B, \, c \in C \big\rangle}$ is an algebra ideal. Equivalently, the composition map of the twisting map $\ta^{-1}$   with the projection $\widetilde\pi_B: A \ot C \to A \otb C$ has to satisfy  conditions analogous to \eqref{f-lin}. This needs not in general follow from $I$ being an algebra ideal.

\begin{lem}
Suppose $A$ and $C$ are $*$-algebras, with $B$ a $*$-subalgebra of $A$. Let $\ta$, satisfying  \eqref{C1} and \eqref{C2}, be 
such that $* = \ta (*_A \otimes *_C)\circ \textup{flip}$ gives a $*$-structure on the twisted tensor product algebra $C \ot^\ta \!A$ (see Lemma \ref{lem:*inv}). 
Let  $\ff:B \to C$ be a $*$-algebra morphism.  The $*$-structure   on   $C \ot^\ta \!A$, $(c \ot a)^*=\ta(a^* \ot c^*)$, 
descends to a well-defined $*$-structure on the  twisted push-forward algebra $\pf$ 
if and only if  $\tap$ is also inner $B$-linear:
\beq\label{innerB}
\tap (a b \ot c)= \tap (a \ot \ff(b)c), \quad \forall \, a \in A, \; b\in B, \; c \in C.
\eeq
\end{lem}
\begin{proof}
Condition \eqref{innerB} is obtained directly by imposing that $(c \ot ba)^*$  equals $(c \, \ff(b) \ot a)^*$ in $C \otb A$, or equivalently that $I$ is a $*$-ideal of $C \ot A$.
\end{proof}
\begin{rem}
The condition \eqref{innerB} is equivalent to $\ta(J) \subseteq I$, with  the ideals $I,J$ given as above.  
When  $J \subset A \ot C$ is an algebra ideal and   $\ta(J) = I$, then 
$\ta$ descends to a well--defined algebra isomorphism 
$
\ta: (A\otb^{\tain} \! C \, , m^{\ta^{-1}})  \longrightarrow (C \otb^\ta  \! A  \,, m^{\ta})
$.
 \end{rem}
 
\subsection{From twisted tensor product algebras to push-forward Hopf--Galois extensions}\label{sec:pf-HG}

We now apply the results of the previous section to the case of $B \subset A$ being a Hopf--Galois extension. As mentioned before, the   push-forward  $C \ot_B A$ of $A$ along an  algebra morphism  $\ff: B \to C$  plays the role of a geometrical pull-back of a principal bundle.
In the framework of  central Hopf--Galois extensions - that is under the hypothesis that the algebra $B$ is central in $A$, and thus in particular commutative, and that the image $\ff(B)$ of $\ff$  is central in $C$ --
it was shown in \cite[Thm. 1.4]{kassel} that the algebra $C \ot_B A$ with the tensor product algebra multiplication is an $H$-Galois extension of $C$. Moreover it is faithfully flat as $C$-module if $A$ is such as a $B$-module. We now prove analogous results in the more general context of twisted tensor product algebras, without centrality assumptions. 

Assume   that the algebra  $A$ is a right $H$-comodule algebra for a Hopf algebra $H$,  with
coaction $\delta: A \to A \ot H$, $a \to \zero{a} \ot  \one{a}$, and $B$ is the subalgebra 
of coinvariant elements, $B=A^{coH}=\{ b \in B \, | \, \delta(b)= b \ot 1_H\}$.  Moreover assume $A$ is an 
$H$-Galois extension of $B$.  
Motivated by the geometric pull-back interpretation, consider $C$ as a trivial $H$-comodule and  $C \ot A$ as a right $H$-comodule via the tensor product coaction 
$\id_C \ot \delta$.  Since $B=A^{coH}$ is the subalgebra of coinvariants, the $H$--comodule structure descends to  $C \ot_B A$, with coaction $\id_C \otb \delta$.  Similarly,  consider the vector space $A \ot C$   as a right $H$-comodule with the tensor product coaction
 $(\id_A \ot  \textup{flip})\circ (\delta \ot \id_C)$. In the following, we fix an algebra map $\ff: B \to C$. 
 \begin{prop}\label{prop:Hcas}
Let $\ta:A\ot C\to C\ot A$  be a   twisting map  such that the projection $\tap$ is a $B$-bimodule morphism as in \eqref{f-lin}. Suppose $\ta$ is an $H$-comodule morphism, that is
$$
  \taa{c} \ot \zero{(\taa{a})} \ot   \one{(\taa{a})} = \taa{c} \ot \taa{(\zero{a})} \ot   \one{a}  \; , \qquad \forall \, a \in A, \, c \in C 
$$
then the 
 twisted push-forward algebra $\pf$ is an $H$-comodule algebra  with coaction  $\id_C \ot_B \delta$.
 \end{prop}
\begin{proof}
We have to show that the coaction  $\id_C \otb \delta$ is an algebra map with respect to the twisted product $m^\ta$. 
For $a \in A$, $c \in C$, using that $\delta$ is an algebra map,  we compute
\begin{align*}
(\id_C \otb \delta) \big((c \otb a) \dta (c' \otb a')  \big)
&= (\id_C \otb \delta) \big(c \taa{c'} \otb \taa{a}  a'   \big)
\\
&=  c \taa{c'} \otb \zero{(\taa{a})} \zero{ a' } \otb \one{(\taa{a})} \one { a'}  \,.
\end{align*}
On the other hand (writing again $\dta$, by abuse of notation)
\begin{align*}
\big((\id_C \otb \delta) (c \otb a) \big)  \dta \big( (\id_C \otb \delta)  (c' \otb a')  \big)
&= \big(c \otb \zero{a} \otb \one{a} \big)  \dta \big(  c'  \otb \zero{a'} \otb \one{a'} \big)
\\
&= c \taa{c'} \otb \taa{(\zero{a})} \zero{a'} \otb   \one{a} \one{a'}
\end{align*}
and the two expressions coincide since $\ta$ is an $H$-comodule morphism.
\end{proof}
The algebra $C$ is contained in the  subalgebra of $\pf$ consisting of coinvariant elements, 
 $ C \subset \big( \pf \big)^{coH}$, $c \mapsto c \otb 1_B$.
 With this identification one has 
 \beq\label{}
c' (c \otb a) = c' \tap(1_A \ot c) a \, \, ,  \quad   (c \otb a) c' = c \tap(a \ot c')   \; ,\quad a \in A, \, c,c' \in C
\eeq
and one can form the balanced tensor product of $\pf$ with itself over its subalgebra
 $C$. We make the assumption that the map $\tap= \pi_B \circ \ta$ satisfies (cf. \eqref{nor})
 \beq\label{nor-B} 
\tap  (1_A \ot c)= c \otb 1_A \;, \qquad  \tap (a \ot 1_C)=1_C \otb a.
\eeq
Then  the balanced tensor product of $\pf$ with itself over  $C$ is given by 
$$
 (\pf) \ot_C  (\pf) = (\pf) \ot  (\pf) / \mathcal{I}
$$
with ideal  
 \beq
\mathcal{I} = \big\langle (c \otb a ) \ot (c'' c' \otb a' ) - (c \taa{c''} \otb \taa{a}) \ot  (c' \otb a' )\, ; \quad  
 a,a' \in A, \, c,c', c'' \in C \big\rangle .
 \eeq
Notice that in the quotient, 
\beq\label{cta=ac}
\taa{c} \otb \taa{a} \ot_C  1_C \otb 1_A = 1_C \otb a \ot_C c \otb 1_A \; , \quad \forall a \in A, \, c \in C.
\eeq

 \begin{lem}\label{lem:CAB}
Let  $\ta:A\ot C\to C\ot A$ be a twisting map which is a right $H$-comodule morphism. Assume $\tap$ is a $B$-bimodule morphism which satisfies \eqref{nor-B} and  the distributive law 
$\tap(a \ot cc') = \taa{c}  \tap(\taa{a} \ot c')$ (cf.~\eqref{C2el}).
 Then $ (C \ot_B A) \ot_C  (C \ot_B A)$ and $ C \ot_B (A \ot_B A)$
 are isomorphic as left $C$-modules and right $H$-comodules  via the map
 \beq\label{CAA}
 (C \ot_B A) \ot_C  (C \ot_B A) \to C \ot_B (A \ot_B A) \, , \qquad (c \otb a) \ot_C (c' \otb a') \mapsto c \taa{c'} \otb \taa{a} \otb a' 
 \eeq
with inverse 
 $c \otb (a  \otb a') \mapsto  (c \otb a) \ot_C (1_C  \otb a')$.
  \end{lem}
  \begin{proof}  
  The map 
\begin{align*}
\alpha: &(C \ot_B A) \ot  (C \ot_B A) \to C \ot_B (A \ot_B A),
\\
&(c \otb a) \ot (c' \otb a') \; \mapsto \; c \tap(a \ot c') \otb a'= c \taa{c'} \otb \taa{a} \otb a' 
\end{align*}
is well-defined since $\tap= \pi_B \circ \ta$   is a $B$-bimodule morphism, that is it satisfies \eqref{f-lin}. For $c'' \in C$, and property $\tap(a \ot cc') = \taa{c}  \tap(\taa{a} \ot c')$ of the map $\tap$, one has
 \begin{align*}
 \alpha \big( (c \otb a) \ot (c'' c' \otb a') \big) &=  c \taa{(c'' c')} \otb \taa{a} \otb a'  
= c \taa{c'' } \taadue{c'} \otb \taadue{(\taa{a})} \otb a' 
\\
&=  \alpha \big( (c \taa{c''}  \otb \taa{a})\ot  (c' \otb a') \big) .
 \end{align*}
Thus $\alpha$ vanishes on the ideal $\mathcal{I}$ and descends to the well-defined map \eqref{CAA} on the balanced tensor product $(C \otb A) \ot_C  (C \otb A)$.
The map  $C \ot_B (A \ot_B A) \to (C \ot_B A) \ot_C  (C \ot_B A)$, $c \otb (a  \otb a') \mapsto  (c \otb a) \ot_C (1_C  \otb a')$ is well-defined. By direct check one verifies that it provides  a right inverse of $\alpha$, due to property \eqref{nor-B} of $\tap$, and a left inverse, due to identity \eqref{cta=ac} in the quotient.  
Both maps are morphisms of left $C$-modules and right $H$-comodules.
  \end{proof}

  The result of \cite[Thm.1.4]{kassel} is generalised to possibly non central $H$-Galois extensions by the following result.

\begin{thm}\label{thm:pushHG}
Let $B \subset A$ be an $H$-Galois extension with $A$ faithfully flat as a $B$-module. Let  
$\ta:A\ot C\to C\ot A$ be a twisting map with properties as in Lemma \ref{lem:CAB}. Then 
the twisted push-forward algebra $\pf$ of $B \subset A$ along the map $\ff:B \to C$ is a faithfully flat $H$-Galois extension of $C$.
\end{thm}
\begin{proof}  
The left $C$-module $C \ot_B A$ is faithfully flat since  $A$ is faithfully flat as a left $B$-module, \cite[Prop. 5 Ch.I \S 3]{bour}. 
Consider the map 
$$
\chi^\ta:  (\pf) \ot_C  (\pf) \to  (\pf) \ot H\, , \, \, \, 
(c \otb a) \ot_C (c' \otb a') \mapsto (c \otb a)  \dta (c' \ot_B   \zero{a'}) \ot \one{a'}.
$$
With the  isomorphism \eqref{CAA}, the  map $\chi^\ta$   is simply  
 $\chi^\ta= \id_C \otb \chi$, with $\chi: A \otb A \to A \ot H$ the canonical map of the algebra extension $B \subset A$. It follows that $\chi^\ta$ is surjective.
 Since the left $C$-module $C \ot_B A$ is faithfully flat and $\chi^\ta$ is surjective, by \cite[Lemma 4.2]{tak77}  one has $C=(\pf)^{coH}$.  Then  $C \subset \pf $ is an $H$-Galois extension, with canonical map
$\chi^\ta$.
 \end{proof}
 
We give an explicit formula for the translation map of the extension obtained through the  push-forward construction. 
\begin{lem}
Let $\tau: H \to A \ot_B A$, $h \mapsto \tone{h} \ot_B \ttwo{h}$ be the translation map of the   $H$-Galois extension $B \subset A$. 
 Then the  translation map $\tau^\ta= (\chi^\ta)^{-1} \vert_{1_C \otb 1_A \ot H}$ of the $H$-Galois extension $\pf$ is given by
\beq
\tau^\ta: H \to (\pf) \ot_C  (\pf), \quad h \mapsto (1_C \ot_B \tone{h} )\ot_C (1_C \ot_B \ttwo{h}) . \label{tau-pf}
\eeq  
\end{lem}
 \begin{proof} 
 We compute:
 \begin{align*}
(m^\ta \ot_C \id_{C \ot_B A})& (\id_{C \ot_B A} \ot \tau^\ta)  \chi^\ta \Big( (c \otb a) \ot_C (c' \otb a') \Big) 
\\
&
=  (m^\ta \ot_C \id_{C \ot_B A})\Big(  c \taa{c'} \otb   \taa{a}  \zero{a'} \ot \tau^\ta (\one{a'})  \Big)  
\\
&
=  \Big(  c \taa{c'} \otb   \taa{a}  \zero{a'}  \Big) \dta  \Big(  1_C  \otb   \tone{\one{a'}}    \Big)  \ot_C  \Big(  1_C  \otb   \ttwo{\one{a'}}    \Big) 
\\
&
=  \Big(  c \taa{c'} \otb   \taa{a}  \zero{a'}     \tone{\one{a'}}    \Big)  \ot_C  \Big(  1_C  \otb   \ttwo{\one{a'}}    \Big) 
\\
\intertext{being $\tap$ normal. 
Due to \eqref{cta=ac}, this is equal to}
&
= c \otb a \zero{a'}     \tone{\one{a'}} \ot_C c' \otb \ttwo{\one{a'}} 
\\
&
=c \otb a\ot_C c' \otb a'.
\end{align*}
with the identity $\zero{a} \tone{\one{a}} \ot_B \ttwo{\one{a}} = 1 \ot_B a$ for $a \in A$ (from $\tau$ being the translation map of the extension $B \subset A$).
The other identity, $\chi^\ta \left( \tau^\ta (h) \right) = 1_c \otb 1_A \ot h$, follows from the definition of the two maps and the analogous identity $\chi \left( \tau (h) \right) = 1 \ot h$ for the maps $\chi, \, \tau$ of the Galois extension $B \subset A$. 
\end{proof}

 \begin{lem}\label{lem:cleft}
In the hypothesis of  Theorem \ref{thm:pushHG}, if the $H$-Galois extension $B \subset A$ is cleft, respectively trivial, then the extension $C \subset \pf$ is cleft, respectively trivial.
 \end{lem}
 \begin{proof}
 Recall that an $H$-extension is cleft provided there exists a cleaving map: a convolution invertible $H$-comodule morphism $\gamma: H \to A$, for $H$ with coaction given by its coproduct. If $\gamma$ is also an algebra map, the extension is said to be trivial. 
 Suppose that $B \subset A$ is cleft with cleaving map $\gamma: H \to A$. 
 Then the map   $\pi_B \circ u_A  \circ \gamma $ is a cleaving map for $C \subset \pf$, where  
 $u_A$ is the algebra morphism $u_A  : A \to C \ot^\ta \!A \, , \, a \mapsto 1_C \ot a$, see \eqref{projpi_alg}.
Indeed, it is an $H$-comodule morphism being the composition of two such maps. Moreover, it is convolution invertible with inverse $u_A \circ \bar{\gamma}$, for $\bar{\gamma}$ the convolution inverse of $\gamma$, as  shown by using that $\psi$ 
satisfies \eqref{nor-B}.
In particular if $\gamma$ is an algebra map, giving $B \subset A$ as a trivial extension, the map $u_A \circ  \gamma$ is an algebra map as well and $C \subset \pf$ is trivial.
 \end{proof}
 \begin{rem}  
 The paper \cite{Dur-un} presents a noncommutative construction of classifying spaces for compact matrix quantum groups and introduces a quantum analogue of the classical classifying map. 
 The approach there relies on a generalized crossproduct construction, akin to the one used in our paper, applied   to  unital $*$-algebras generated by elements $\psi_{ki}$
with the sole relations  $\sum_k \psi_{ki}^* \psi_{kj} = \delta_{ij}$. 
  \end{rem} 

\section{Examples}\label{sec:cases}

In this section we analyse some particular cases of push-forward of Hopf--Galois extensions $B \subset A$: the case of Galois objects (\S \ref{sec:Gobjects}) and those of push-forwards along the (injective) algebra inclusion $i: B \to A$ (\S \ref{sec:pf-tot}) and along a (surjective) projection map $\pi: B \to B/I$ to a quotient algebra (\S \ref{sec:quotient}).

\subsection{The push-forward of Galois objects}\label{sec:Gobjects}
In the case of an $H$-Galois object, $B=\kk$, for any unital algebra $C$, the  only  possible algebra map $\ff: \kk \to C$ is the unit $\eta_C$ of $C$. For any twisting map $\ta: A \ot C \to C \ot A$, the compatibility condition \eqref{f-lin} between $\ta$ and  $\eta_C$ is automatically satisfied, as it  simply corresponds to  the  linearity of $\ta$. Moreover, 
 $\tap=\ta$.  Then the requirements on the map $\tap$ in  Lemma \ref{lem:CAB}, and subsequently in Theorem \ref{thm:pushHG}, reduce to 
the twisting map $\ta$ being normal and satisfying the distributive law \eqref{C2}. The latter property is redundant as it already follows from the normality condition for the twisting map $\ta$ (which also implies \eqref{C1}). 

As a consequence of Theorem \ref{thm:pushHG} we have the following result.

\begin{cor}\label{cor:obj-G}
Let $A$ be an $H$-Galois object with $A$ faithfully flat as a $\kk$-module. For any normal
 twisting map $\ta:A\ot C\to C\ot A$  which is   an $H$-comodule morphism,  the twisted tensor product algebra $C \ot^\ta \!A$ is a faithfully flat $H$-Galois extension of $C$.
\end{cor}

We are in the context, treated in \cite{kassel}, of push-forwards of central Galois extensions. Here $\psi$ could be more general that simply the $\textup{flip}$ map, thus leading to a multiplication on $C\ot A$ different from the usual tensor product multiplication.

\begin{lem}\label{lem:double}
Let $\{C, A\}$ be a matched pair of bialgebras (see Example \ref{ex:smash}). Let $A$ be an $H$-Galois object, with coaction $\delta: A \to A \ot H$,  and $A$ faithfully flat as a $\kk$-module.  Suppose the comultiplication $\Delta_A$
of $A$ is an $H$-comodule morphism, for $A \ot A$  with coaction $\id \ot \delta$, and that the action $\triangleleft$ of $C$ on $A$ is an $H$-comodule morphism, for $A \ot C$ with tensor coaction $\delta^\ot= \id \ot \delta$   as before. Then 
 the double crossed product algebra $C \bowtie A$ is a faithfully flat $H$-Galois extension of $C$.
\end{lem}
\begin{proof}
In view of Corollary \ref{cor:obj-G}, we only have to show that $\ta$ in \eqref{ta-double} is an $H$-comodule morphism.
Under the hypothesis that the comultiplication $\Delta_A$ is an 
$H$-comodule morphism,  
$$
\one{(\zero{a})} \ot \two{(\zero{a})} \ot \one{a} = \one{a} \ot \zero{(\two{a})} \ot \one{(\two{a})} \, , \; \forall a \in A
$$
and that the action $\triangleleft$ is an $H$-comodule morphism,  
$$
\delta(a \triangleleft c)= (\zero{a} \triangleleft c)\ot \one{a} \, ,   \; \forall a \in A,\, c \in C,
$$
for the twisting map $\ta$ in \eqref{ta-double}  we compute,
\begin{align*}
\delta^\ot \left( \ta(a \ot c) \right) 
&= (\one{a} \triangleright \one{c}) \ot \delta(\two{a} \triangleleft \two{c})
= (\one{a} \triangleright \one{c}) \ot (\zero{(\two{a})} \triangleleft \two{c})\ot \one{(\two{a})}
\\
&= (\one{(\zero{a})} \triangleright \one{c}) \ot ( \two{(\zero{a})}  \triangleleft \two{c})\ot \one{a}
= \ta(\zero{a} \ot c) \ot \one{a}
\\
&= (\ta \ot \id) \, \delta^\ot(a \ot c).
\end{align*}
Thus $\ta$ is an $H$-comodule morphism.
\end{proof}

\begin{rem}
Twisted tensor product algebras are studied as Galois objects for bicrossed product Hopf algebras in \cite[\S 4]{BiGa24}.
Given an $H$-Galois object $A$  and a $U$-Galois object $R$, for $\{H, U\}$ a matched pair of Hopf algebras, a twisted tensor product algebra $A \ot^{\ta} \!R$ is shown to be a Galois object for the bicrossed product Hopf algebra $H \bowtie U$.  
\end{rem}

\begin{ex} 
Let $q$ be a  primitive $n$-th root of unity contained in $\kk$, with $n \ge 2$.  
The \emph{Taft algebra} $T_n(q)$ is the $n^2$-dimensional Hopf algebra generated by two elements $g$ and $x$ subject to the relations
$$
g^n = 1, \quad x^n = 0, \quad xg = q\,gx.
$$
It has comultiplication, counit, and antipode defined by
$$
\Delta(g) = g \otimes g, \quad \varepsilon(g) = 1, \quad S(g) = g^{-1},
$$
$$
\Delta(x) = 1 \otimes x + x \otimes g, \quad \varepsilon(x) = 0, \quad S(x) = -x\,g^{-1} .
$$
For each  $s \in \kk$, define $A_s$ to be the  algebra  generated by elements $G$ and $X$ subject to
\beq\label{com-taft}
G^n = 1, \quad X^n = s, \quad X G = q\,G X.
\eeq
Then $A_s$ is  a right $T_n(q)$-comodule algebra with coaction $\delta: A_s \to A_s \otimes T_n(q)$ defined on generators by
$$
\delta(G) = G \otimes g, \quad \delta(X) = X \otimes g + 1 \otimes x
$$
with $\kk$ the subalgebra of coinvariants.  The algebra $A_s$ is a Galois object, \cite{sch04}.  
A vector space basis of $A_s$  is given by  the elements $G^a X^b$ with $a,b$ integers, $0 \leq a,b < n$.
The coaction    on these basis elements  is  
$$
\delta(G^a X^b ) = \sum_{k=0}^{b} \qbin{b}{k}_q \, G^a X^k  \otimes g^{\,a+k} x^{\,b-k} ,
$$
where $\qbin{b}{k}_q$ denotes the $q$-binomial coefficient, as can be shown by induction.

Let $C$ be an algebra and suppose there exists a  normal twisting map $\ta: A_s \ot C \to C \ot A_s$. 
This can be specified on the algebra generators of $A_s$ and then extended by requiring \eqref{C1el}, \eqref{C2el} are satisfied and compatible with the relations \eqref{com-taft}. The most generic linear map $\ta: A_s \ot C \to C \ot A_s$ is of the form
$$
\ta(G \ot c):= \sum_{a,b} c_{ab} \ot G^a X^b\, , \quad  \ta(X \ot c):= \sum_{a,b} d_{ab} \ot G^a X^b \, , 
$$
on the algebra generators, where  $c_{ab}=c_{ab}(c)$ and $d_{ab}= d_{ab}(c)$ both belonging to $C$. 
In order for this map to be an $H$-comodule morphism, we must have 
$c_{ab}=0$  for $a>1$ or $b\geq1$ and $d_{ab}=0$  for $a>1$ or $b>1$. Moreover $c_{00}=0$ and $d_{00}=d_{11}=0$, $d_{01}=c$.
Thus $\ta$ is determined by two linear maps $\alpha, \beta : C \to C$ as
$$
 \ta(G \ot c)= \alpha(c) \ot   G \, , \quad  \ta(X \ot c)= c  \ot X + \beta(c) \ot G \, , \quad  
$$
with $\alpha(1)=1$ and $\beta(1)=0$ in order for $\ta$ to be normal.
It is extended by requiring \eqref{C1el}, \eqref{C2el} are satisfied, so to get a twisting map, thus in particular 
$$
\ta(G^m \ot c) = \alpha^m(c)  \ot G^m \, , \quad
\ta(X^m \ot c) = c \ot X^m + \sum_{k=0}^{m-1} \qbin{m}{k}_{q} \, \beta^{m-k} (c) \ot  G^{m-k} X^k$$
(where $\alpha^m$ and $\beta^m$ are the $m$-th iterations of the maps) 
for each $m \in \mathbb{N}$, together with
$$
\ta(GX \ot c)=   \alpha(c)  \ot G X  + \alpha(\beta(c)) \ot G^2 \, , \quad 
\ta(XG \ot c)= \alpha(c)  \ot X G + \beta (\alpha(c)) \ot G^2. 
$$
The compatibility with  \eqref{com-taft} then requires 
$$
\alpha^n  = \id \, , \quad q\, \alpha \circ \beta  = \beta \circ \alpha  \, , \quad \beta^{n-k}=0 \, , \; \forall k=0, \dots, n-1.
$$
Summing up we have that normal twisting maps $\ta: A_s \ot C \to C \ot A_s$ which are $H$-comodule morphisms are in bijection with unital algebra maps $\alpha : C \to C$ such that $\alpha^n  = \id$:
$$
\ta (G^a X^b \ot c ) = \alpha^a (c) \ot G^a X^b .
$$
This leads to the algebra $C \ot^\ta \!A_s$ with product
$$
(c' \ot G^{a_1} X^{b_1}) \cdot^\ta (c \ot G^{a_2} X^{b_2}) = q^{b_1 a_2 }c' \alpha^{a_1} (c) \ot G^{a_1 + a_2 } X^{b_1 + b_2}
$$
and the   $T_n(q)$--Galois extension $C \subset C \ot A_s$.

 Examples of such a map $\alpha$ are $\alpha_{(m)}: c \mapsto q^m c$, $m \in \mathbb{N}$, for any algebra $C$, or
$\alpha$ a suitable permutation of generators $c_j$ for $C=\kk[c_1, \dots , c_n]$. 
For   $A$ being the Hopf algebra $T_n(q)$, examples come from  twisting maps as in \eqref{ta-double} and double crossed products  $C \bowtie T_n(q)$  for which the action $\triangleleft$ of $C$ on $T_n(q)$ is an $H$-comodule morphism (Lemma  \ref{lem:double}).
The particular case of Hopf algebras which factor through  two Taft algebras
is studied in \cite{Ag18} where the corresponding bicrossed products are classified. 
 \end{ex}

\subsection{The push-forward to the total algebra}\label{sec:pf-tot}

In parallel with the classical geometric result that the  pull-back of a principal bundle to the total space  via the bundle projection   is trivial, we construct the push-forward $A \otb^{\ta_\ell} \!A$ of a faithfully flat Hopf--Galois extension $B \subset A$ to the total space algebra along the inclusion $i : B \to A$ and show it is trivial. Here the twisting map $\ta_\ell$ is constructed using the strong connection $\ell$ of the extension.

Let $B\subset A$ be a faithfully flat $H$-Galois extension, for a Hopf algebra $H$ with bijective antipode and 
coaction $\delta: A \to A \ot H$, $\delta(a) = \zero{a} \ot \one{a}$.  Then, from \cite[Thm. 5.6]{scsc}, there always exists a strong connection in the sense of \cite{dgh}. 
This can be given as a linear map
$\ell: H \to A \ot A$ with the following properties (see \cite[Lemma 2.3]{BH04}),
\beq\label{ell} 
\ell(1) =1 \ot 1 , \quad 
\pi_B \circ \ell = \tau , \quad
(\id_A \ot \delta)\circ \ell = (\ell \ot \id_H )\circ \Delta  , \quad
(\delta_l \ot \id_A)\circ \ell  =(\id_H \ot \ell)\circ \Delta .
\eeq
Here $\tau= \chi^{-1} \vert_{1 \ot H}: H \to A \ot_B A$ is the translation map of the extension, 
$\pi_B : A \ot A \to A \ot_B A$ the canonical projection and $\delta_l:A \to H \ot A, \; a \mapsto S^{-1} (\one{a}) \ot \zero{a}$ the induced left coaction of $H$ on $A$.  The map $\ell$ is a lift to $A \ot A$ of the translation map of the extension.
We use the notation  $\ell(h)= \lone{h} \ot \ltwo{h}$ and $\tau(h)= \tone{h} \otb \ttwo{h}$  for the images of an element $h \in H$. Then the last two identities in \eqref{ell} read
\begin{align}
\label{prop-ell0}
\lone{h} \ot \zero{ \ltwo{h}} \ot \one{\ltwo{h}} = \lone{\one{h}} \ot  \ltwo{\one{h}} \ot \two{h} \in A \ot A \ot H
\\
\label{prop-ell-S}
S^{-1} (\one{\lone{h}} ) \ot \zero{\lone{h}} \ot \ltwo{h} = \one{h} \ot \lone{\two{h}} \ot  \ltwo{\two{h}}  \in H \ot A \ot A .
\end{align}

Using a strong connection $\ell$ we construct a twisted tensor product on $A \ot A$ whose associativity 
requires a condition on the map $\ell$. However this is not restrictive for the associativity of the product in the push-forward quotient algebra 
$A \otb^{\ta_\ell} A$ (see Remark~\ref{rem-strong}).
 
\begin{prop}\label{prop:ta-ell}
Let $H$ be a Hopf algebra with bijective antipode and $B\subset A$ a faithfully flat $H$-Galois extension. Assume it admits a strong connection $ \ell$ such that 
 \beq\label{prop-ell}
 \ell(k h)=\lone{h} \lone{k}\ot \ltwo{k} \ltwo{h} ,\quad  \forall \, h,k \in H.
 \eeq
Then the map 
\beq\label{psiell}
\ta_\ell: A \ot A \to A \ot A \, , \quad a \ot c \mapsto \zero{a} c \lone{\one{a}} \ot \ltwo{\one{a}}
\eeq
is a twisting map. 
\end{prop}
\begin{proof}
To prove that the corresponding product $m^{\ta_\ell}$ is associative we show that condition \eqref{Oeq} is satisfied. For each $a, c, a', c' \in A$, for the left hand side of \eqref{Oeq}
we compute
\begin{align*}
  (\id \ot m)&\circ (\ta_\ell \ot \id)\circ (\id  \ot m \ot \id)	\circ (\id \ot \id  \ot \ta_\ell) 
(a \ot c \ot a' \ot c') 
\\
&=(\id \ot m)\circ(\ta_\ell \ot \id)(a \ot c     \zero{a'} c' \lone{\one{a'}} \ot \ltwo{\one{a'}}) 
\\
&= \zero{a} c     \zero{a'} c' \lone{\one{a'}}  \lone{\one{a}} \ot \ltwo{\one{a}}  \ltwo{\one{a'}}.
\end{align*}
For the right hand side of \eqref{Oeq} we compute 
\begin{align*}
(m  \ot \id  )&\circ(\id  \ot \ta_\ell)\circ(\id  \ot m  \ot \id)\circ(\ta_\ell \ot \id \ot \id)
(a \ot c \ot a' \ot c') 
\\
&= (m  \ot \id  )\circ(\id  \ot \ta_\ell) (\zero{a} c \lone{\one{a}} \ot  \ltwo{\one{a}} a' \ot c') 
\\
&=  \zero{a} c \lone{\one{a}}    \zero{\big(  \ltwo{\one{a}} a'  \big)}  c' \,
\lone{\big( \one{ ( \ltwo{\one{a}} a' )} \big) } \ot \ltwo{\big( \one{( \ltwo{\one{a}} a' )} \big)} 
\\
&=  \zero{a} c \lone{\one{a}}    \zero{ (  \ltwo{\one{a}} ) } \zero{ a'  }  c' \,
\lone{ \one{ \big( (  \ltwo{\one{a}} ) } \one{ a'  }\big)} \ot \ltwo{\big( \one{ (  \ltwo{\one{a}} ) } \one{ a'  }\big)} 
\\
\intertext{being $\delta$ an algebra map, and }
&= \zero{a} c     \zero{a'} c' \lone{(\one{a} \one{a'})}   \ot   \ltwo{(\one{a} \one{a'})} 
\end{align*}
being $\tone{h} \zero{(\ttwo{h} )} \ot \one{(\ttwo{h})}  =1 \ot h$,  for all $h \in H$ (following from the fact that the translation map gives a right inverse of  $\chi$).
The  expressions  for the l.h.s and r.h.s. then coincide due to the property \eqref{prop-ell}
of the map $\ell$. 
\end{proof}

The twisting map  $\ta_\ell$ is clearly right normal, $\ta_\ell (1_A \ot c) = c \ot 1$, for each $c\in A$, while in general it is not left normal. Indeed 
$$\ta_\ell(a \ot 1_A)= \zero{a} \ell(\one{a}) = s(a) \in B \ot A,$$
 where  $s: A \to B \ot A$ is the (unital left $B$-module and right $H$-comodule) splitting of the multiplication map  $B \ot A \to A$ associated to the strong connection $\ell$ (see \cite[Lemma 2.3]{BH04}). 
The associative algebra $A \ot^{\ta_\ell} A$  with  multiplication
\beq \label{psi-ell-am}
(a \ot c) \cdot_{\ta_\ell} ( a' \ot c')= a \zero{a'} c \lone{(\one{a'})} \ot \ltwo{ (\one{a'})} c' 
\eeq
will then not-necessarily be unital.
The distributive law  \eqref{C2el}  is implied by the fact that $\ta$ is a right normal twisting map. Even if not left normal in general, $\ta$ also satisfies \eqref{C1el}:
\begin{align*}
\ta_\ell(aa' \ot c) &=  \zero{a} \zero{a'} c \, \lone{(\one{a} \one{a'})} \ot \ltwo{(\one{a} \one{a'})}
=  \zero{a} \zero{a'} c \, \lone{\one{a'}} \lone{\one{a}} \ot \ltwo{\one{a}} \ltwo{\one{a'}}
\\
&= \ta_\ell \left( a \ot \zero{a'} c \lone{\one{a'}} \right) \ltwo{\one{a'}}
=
\ta_\ell \left( a \ot c^{[\ta_\ell]} \right)  {a'}^{[\ta_\ell]} .
\end{align*} 

Even though the property \eqref{prop-ell} of the strong connection does not hold in  general, there are examples in which the formula \eqref{prop-ell} still allows one to construct $\ell$ (e.g. the Hopf bundle in \cite{gw}, described in \S \ref{sec:inst}). For finitely generated algebras, once $\ell$ is defined on the generators of $H$, one can  check whether the formula extends $\ell$ consistently to all of $H$.
If the algebras $H$ and $A$ are commutative, the formula is compatible with the commutativity of the algebras (although nothing can be said, in general, about its compatibility with any additional relations in $H$).
The property \eqref{prop-ell}  is preserved under $2$-cocycle deformations, as shown in Appendix \ref{app:B}. 

The covariant $i$-extension of $A$ along the algebra inclusion $\ff=i : B \to A$, in the sense of \eqref{push}, is the  
  $A$-bimodule given by the balanced tensor product  $A \ot_B A$.  
The twisted multiplication $m^{\ta_\ell}$ of $A \ot^{\ta_\ell} A$ descends to it. 

\begin{prop}\label{prop:ta-tau}
Let $B\subset A$ be a faithfully flat  $H$-Galois extension.  
The multiplication $m^{\ta_\ell}$ of the twisted tensor algebra $A \ot^{\ta_\ell} \!A$, defined via the twisting map in \eqref{psiell}, descends to a well-defined algebra structure 
on the covariant $i$-extension $A \otb A$ of $A$ along the algebra inclusion $\ff=i : B \to A$.  The resulting algebra 
$A \otb^{\ta_\ell}  \!A$ is unital.
\end{prop}
\begin{proof}
By Proposition \ref{prop:Blin}, it suffices to show that condition \eqref{f-lin} holds; that is, that the 
 composition $\tap_\ell= \pi_B \circ \ta_\ell$  of the twisting map $\ta_\ell$ with the projection $\pi_B: A \ot  A \to A \otb A$ is 
 $B$-linear,  $\tap_\ell (b a \ot a' b') = b \tap_\ell (a \ot a') b'$. 
Using as above  $\tau(h)=\tone{h} \otb \ttwo{h}$ for  the image of an element $h \in H$ under the translation map, we have
 $\tap_\ell=\ta_\tau$ where 
\beq\label{psitau}
\ta_\tau: A \ot A \to A \otb A \, , \quad a \ot c \mapsto \zero{a} c \tone{\one{a}} \otb \ttwo{\one{a}}.
\eeq
The map $\ta_\tau$ is  left $B$-linear being $B=A^{coH}$. The right $B$-linearity follows from the property of the translation map, 
 $b \, \tau(h)= \tau (h) b$, for any $b\in B$, $h\in H$  (this is derived from the injectivity of the map $\chi$).

The resulting algebra $A \otb^{\ta_\ell}  A$ is unital. The right normality of $\ta_\ell$  gives  $\ta_\tau (1 \ot c) = c \otb 1$, for each $c\in A$.
On the other hand,  
$\ta_\tau (a \ot 1)=\zero{a} \tone{\one{a}} \otb \ttwo{\one{a}} =\chi^{-1}\big(\chi (1 \otb a)\big)= 1 \otb a$ for each $a \in A$.
Then $1_C \otb 1_A$ is a unit for $A \otb^{\ta_\ell}  \!A$.
\end{proof}

This allows us to construct  the twisted push-forward algebra $A \ot_B^{\ta_\ell} A$  of $B \subset A$ along the algebra inclusion $\ff=i : B \to A$ with product
\beq \label{psi-ell-am2}
(a \ot_B c) \cdot^{\ta_\ell} ( a' \ot_B c')= a \zero{a'} c \tone{ (\one{a'})} \ot_B \ttwo{(\one{a'})} c' = a \zero{a'} c \, \tau(\one{a'})  c' .
\eeq
Given the above expression for the product, we henceforth denote the twisted push-forward algebra by $A \ot_B^{\ta_\tau} \!A$   and write $m^{\ta_\tau}$ for its multiplication.

\begin{rem}\label{rem-strong}
For a strong connection $\ell$ which does not satisfy \eqref{prop-ell}, the multiplication $m^{\ta_\ell}$ may fail to be associative. 
Nevertheless the multiplication $m^{\ta_\tau}$ in \eqref{psi-ell-am2} induced on the balanced tensor product would be associative. Its associativity follows from the analogous property  $\tau(hk)= \tone{k} \tone{h} \otb \ttwo{h} \ttwo{k}$ which the translation map indeed satisfies for all $h,k \in H$ and for any Hopf algebra $H$.
\end{rem}

\begin{rem}
The twisted push-forward algebra  $A \ot_B^{\ta_\tau} \!A$ coincides with the balanced tensor product $A \ot_B A$ with algebra structure of the latter induced by the one of the tensor product algebra $A \ot H$ via the canonical map $\chi: A \ot_B A \to A \ot H$ (which is then an algebra map):
$$
m_{\ta_\tau} = \chi^{-1} \circ m_{A \ot H} \circ (\chi \ot \chi).
$$
Indeed the   map   $\ta_\tau$  satisfies $\ta_\tau (a b \ot c)= \ta_\tau (a \ot bc)$,  for all $a \in A$, $b\in B$ and  $c \in C$, and thus descends to a well-defined map on $A \ot_B A$
 $$\sigma: A \ot_B A \to A \ot_B A \, , \quad a \ot_B c \mapsto \zero{a} c \tone{ (\one{a})} \ot_B \ttwo{(\one{a})}  $$
which is the Durdevi\'c braiding \cite[Eq.~(2.2)]{Durd}. It is invertible 
 with inverse  map
$$\sigma^{-1}: A \otb A \to A \otb A\, , \quad a \ot c \mapsto  \tone{ \big( S^{-1}(\one{c})\big)} \ot \ttwo{\big( S^{-1}(\one{c})\big)} a \zero{c} . $$
The compatibility of $\sigma$ with the multiplication $m_A$ of $A$, analogous to the one in \eqref{Oeq-el}, is shown in \cite[Prop. 2.1]{Durd}.
\end{rem}

The map $\ta_\ell$ is an $H$-comodule morphism since the strong connection $\ell$ is such, see \eqref{ell} (and the same argument applies to the maps $\ta_\tau$ and $\tau$).
 From the general result of Theorem \ref{thm:pushHG}, the twisted push-forward algebra $A \ot_B^{\ta_\tau} \!A$ 
 is an $H$ comodule algebra with $A$ as the corresponding subalgebra of coinvariant elements and the inclusion 
 $A\subset A \ot_B^{\ta_\tau} \!A$ is a faithfully flat $H$-Galois extension. 
 
  \begin{lem}
 The twisted push-forward algebra $A \ot_B^{\ta_\tau} \!A$ of $B \subset A$ along the algebra inclusion $i : B \to A$ is a trivial $H$-Galois extension. 
 \end{lem}
 \begin{proof}
The extension is cleft. A  cleaving map, that is a convolution invertible $H$-comodule algebra 
map $\gamma: H \to A \ot_B^{\ta_\tau} \!A$, is simply given by the translation map of the $H$-Galois extension 
$B \subset A$. It is an algebra map,  being $\chi$ an algebra map for the product $m^{\ta_\tau}$, and has inverse $\overline{\gamma}:=\gamma \circ S$, with $S$ the antipode of $H$. Thus the extension is trivial.
 \end{proof}

\begin{rem}
In formula \eqref{psitau} the element $c$ could be taken in any algebra containing $A$
so that the product of $c$ with elements in $A$ is defined. If $C$ is any algebra with subalgebra $A$ and, as above, $B\subset A$ a faithfully flat $H$-Galois extension, the map 
\beq\label{psitau2}
\ta_\tau: A \ot C \to C\otb A \, , \quad a \ot c \mapsto \zero{a} c \, \tau(\one{a})
\eeq
defines a unital associative product on  the covariant $i$-extension $C \otb A$ of $A$ along the algebra inclusion $i : B \to C$.  The resulting $H$-Galois extension $C \subset C \otb A$ is also trivial and coincides with the push-forward of the trivial $H$-Galois extension $A \subset A \ot_B^{\ta_\tau} \!A$ along the map $i: B \to C$.
\end{rem}

\subsection{The push-forward to a quotient algebra}\label{sec:quotient}
 
In this section, we consider the  push-forward of a Hopf--Galois extension $B \subset A$ along  the  (surjective) projection  map $\ff=\pi: B \to B/I$ from $B$ to its quotient by a two-sided algebra ideal $I$. This provides the algebraic analogue of restricting a principal bundle to a subspace.

For now, let $A$ be an algebra, $B$ a subalgebra of $A$ and $C:=B/I$ be  the quotient algebra of $B$ by a two-sided algebra ideal $I \subset B$ such that $A \cdot I \subseteq I \cdot A$. 
Let  $C \otb A$ be the covariant extension along the projection  $\pi: B \to B/I$.
Although different twisting maps $\ta:A\ot C\to C\ot A$ may exist, their projection $\tap= \pi_B \circ \ta$ onto $C \otb A$ becomes unique when we require it to satisfy \eqref{nor-B} and the $B$-linearity condition \eqref{f-lin};  
this result does not use that $\ta$ is a twisting map but holds for any linear map $\ta:A\ot C \to C\ot A$.   

\begin{prop}
Let $C=B/I$ and $\ta:A\ot C\to C\ot A$ be a linear map. Assume the ideal $I$ is such that
\beq\label{cond-ideal}
A \cdot I \subseteq I \cdot A \, .
\eeq
To satisfy \eqref{f-lin} and \eqref{nor-B}, the map $\tap$ has to be given by 
$$
\tap (a \ot [b])=  1_C \otb a \, b , \qquad  \forall a \in A , \, [b] \in C.
$$
Moreover the map $\tap$  defines an associative unital algebra multiplication on $C \otb A$.
\end{prop}
\begin{proof}
The right $B$-linearity,  together with the second condition in \eqref{nor-B}, uniquely determines the map $\tap$ as 
$$
\tap (a \ot [b])= \tap (a \ot 1) b = 1_C \otb a \, b , \qquad  \forall a \in A , \, [b] \in C.
$$
Condition \eqref{cond-ideal} ensures that the map is well-defined: if
$b,b' \in B$ are such that $\pi(b)=\pi(b')$, that is $b-b' \in I$, then $1_C \ot a \, b = 1_C \ot a \, b' $. Thus 
 $\tap$  does not depend on the representative  of the element $[b]$.
 The left $B$-linearity and  the first condition in \eqref{nor-B} are automatically satisfied by the map 
$\tap$, being $1_C \otb b = [b] \otb 1_A$ for each $b \in B$.

The algebra multiplication  on $C \otb A$ defined  by  $\tap$,
\beq\label{prod-quot}
\big([b] \otb a \big) \cdot^{\tap}  \big([b'] \otb a'\big) = [b] \tap (a \ot [b']) a' =1_C \otb bab'a'
\eeq
is associative. The map $\tap$ satisfies the analogue of the distributive laws \eqref{C1el} and \eqref{C2el}: the
compatibility of $\tap$ with the algebra multiplication $m_A$ is obvious, the one with the multiplication $m_C$ follows from $\pi$ being an algebra map. 
\end{proof}

We now specialize to the case where   $B \subset A$ is an $H$-Galois extension, the map $\tap$ is also an $H$-comodule morphism,
$$
 \taa{[b]} \otb \zero{(\taa{a})} \ot   \one{(\taa{a})} = 1_C \otb \zero{a} b \ot \one{a}=\taa{[b]} \otb \taa{(\zero{a})} \ot   \one{a}  \; , \quad \forall \, a \in A,\,  [b] \in C,
$$
then as a consequence of Theorem \ref{thm:pushHG} we have the following result.
\begin{cor}\label{cor:quot}
Let $B \subset A$ be a faithfully flat $H$-Galois extension and $I$ be a  two-sided algebra ideal $I$ of $B$ which satisfies \eqref{cond-ideal}. Then the (unique)  twisted push-forward algebra $B/I \otb^{\ta} A$  of $B \subset A$ along  $\pi: B \to B/I$ is a faithfully flat $H$-Galois extension of $B/I$.
\end{cor}

We conclude with a characterization of the algebra $\pf$ when the ideal $I$ is finitely generated. 
\begin{lem}\label{lem:IA}
Let $I=\langle y_1, \dots, y_m \rangle$ be the two-sided ideal in $B$ generated by elements $y_1, \dots, y_m \in B$, $m \in \mathbb{N}$. Let $I_A$ be the two sided-ideal in $A$ generated by the elements $y_1, \dots, y_m$.
Then, $C \otb^\ta A$ is isomorphic to the quotient algebra $A \slash I_A$ as an  $H$-comodule algebra.
\end{lem}
\begin{proof}
Define the map 
\beq\label{map-IA}
g: C \otb^\ta A \to A \slash I_A \, , \quad [b] \otb a = 1_C \otb ba \mapsto b a \, .
\eeq
It is well-defined since $[b]=[b'] \iff ba -b'a \in   I_A$. It is an algebra morphism:
$$
g \big( (1_C \otb ba ) \cdot^\ta (1_C \otb b'a' )  \big) = g \big( 1_C \otb ba  b'a'   \big) =  bab'a' =  g(1_C \otb ba )   g(1_C \otb b'a' ).
$$
It is invertible with inverse $g^{-1}:A \slash I_A \to C \otb^\ta A$, $a \mapsto 1_C \otb a$ which is well-defined because $1_C \otb y_j =0$, for $j=1, \dots, m$.
Since the generators $\{y_j \}$ of the ideal belong to the subalgebra of coinvariant elements $B$, the map $g$ in \eqref{map-IA} is also an $H$-comodule morphism.
\end{proof}

\subsubsection{Push-forwards of the $SU(2)$-Hopf bundle to a quotient algebra}\label{sec:inst}

Let us consider the $\mathcal{O}(SU(2))$--Galois extension $\mathcal{O}(S^4_\theta) \subset \mathcal{O}(S^7_\theta)$ introduced in \cite{gw}. 

The $*$-algebra $\mathcal{O}(S^4_\theta)$ is generated by a central real element $x$ and elements
 $\alpha,\beta,\alpha^*,\beta^*$ with   commutation relations
\beq\label{4sphere}
\alpha \beta = \lambda \beta \alpha \; , \quad
\alpha^* \beta^* = \lambda \beta^* \alpha^*\; , \quad
\beta^* \alpha  = \lambda \alpha  \beta^* \; , \quad
\beta \alpha^* = \lambda \alpha^* \beta ,  \qquad \lambda:=e^{2 \pi i \theta} ,
\eeq 
and sphere relation $\alpha^* \alpha + \beta^*\beta + x^2 =1$, for $\theta$ a real parameter.
The   $*$-algebra $\mathcal{O}(S^7_\theta)$   is   generated by elements
$\{z_j, z_j^* \,; \, j=1,\dots, 4\}$ with relations
\beq \label{7sphere}
z_j z_k = \lambda_{j k} \, z_k z_j \; , \qquad
z_j^* z_k = \lambda_{k j} \, z_k z_j^* \; , \qquad
z_j^* z_k^* = \lambda_{j k} \, z_k^* z_j^* \; , \;
\eeq and spherical relation $\sum z_j^* z_j  =1$.
The deformation matrix $(\lambda_{jk})$ is chosen so that the coordinate algebra $\mathcal{O}(SU(2))$ of the (classical) group $SU(2)$ coacts on $\mathcal{O}(S^7_\theta)$ with subalgebra of coinvariants given by $\mathcal{O}(S^4_\theta)$, the coordinate algebra of the deformed $4$-sphere $S^4_\theta$.
The matrix depends on a single deformation parameter $\mu=e^{\pi i \theta}$, with 
 $$
\lambda_{12}= \lambda_{34}=1 \, , \quad 
\lambda_{13}=   \lambda_{24}=\bar \mu  \, , \quad 
\lambda_{14}=\lambda_{23}= \mu\, , \quad  
\lambda_{jk}=\bar \lambda_{kj} .
$$
The right coaction $\delta$ of $\mathcal{O}(SU(2))$ on  $\mathcal{O}(S^7_\theta)$ is defined  on the algebra generators as
\begin{eqnarray}\label{princ-coactSU2}
\delta: \mathsf{u}
 &\longmapsto&
\mathsf{u}
\overset{.}{\otimes}
\mathsf{w}
\;, \quad 
\mathsf{u}=
\begin{pmatrix}
z_1& z_2 & z_3& z_4
\\
-z_2^*  & z_1^* & -z_4^* & z_3^*
\end{pmatrix}^t 
, \quad 
\mathsf{w}=
\begin{pmatrix} 
 w_1 & -w_2^*
\vspace{2pt}
\\
w_2  & w_1^* 
\end{pmatrix} ,
\end{eqnarray}
with $w_1, w_2$ denoting the coordinate functions of $SU(2)$ (and $\overset{.}{\otimes}$   the composition of the tensor product   with matrix multiplication).
The coaction is extended to the whole $\mathcal{O}(S^7_\theta)$ as a
$*$-algebra morphism.
The subalgebra of coinvariant elements in $\mathcal{O}(S^7_\theta)$
is  generated by the entries of 
 the
matrix $p:=uu^*$, and identified with $\mathcal{O}(S^4_\theta)$ via
\beq \label{s4ins7}
\alpha=2(z_1 z_3^* + z_2^* z_4) \; , \quad  \beta=2(-z_1^* z_4 + z_2 z_3^*) \; ,
\quad
x=z_1 z_1^* + z_2 z_2^* - z_3 z_3^* - z_4 z_4^* .
\eeq
With this identification, the commutation relations between the generators of $\mathcal{O}(S^7_\theta)$ and those of its subalgebra $\mathcal{O}(S^4_\theta)$   are given by
\begin{align}\label{comm7-4} 
& z_1 \, \alpha=  \mu \, \alpha \, z_1  \;, \; 
&& z_2 \, \alpha= \bar{\mu} \, \alpha \, z_2  \;, \;
&& z_3 \, \alpha= \mu \, \alpha \, z_3  \;, \;
&& z_4 \, \alpha= \bar{\mu} \, \alpha \, z_4  \;, \;
\nn \\
& z_1 \, \beta= \mu \, \beta \, z_1  \;, \; 
&& z_2 \, \beta= \bar{\mu} \, \beta \, z_2  \;, \; 
&& z_3 \, \beta= \bar{\mu} \, \beta \, z_3  \;, \; 
&& z_4 \, \beta= \mu \, \beta \, z_4  \; 
\end{align}
and analogous ones for $z_j^*$ with $\mu \leftrightarrow \bar\mu$,   while $x$ commutes with all $z_j$, $z_j^*$. 
\medskip

\noindent
\textbf{The push-forward   to the quotient algebra  $\mathcal{O}(S^3_\theta)$}.
Let us  consider the quotient algebra of $B=\mathcal{O}(S^4_\theta)$ by the two-sided algebra ideal generated by the central element $x$: 
$$
C= \mathcal{O}(S^4_\theta) \slash \langle x \rangle  \simeq \mathcal{O}(S^3_\theta) .
$$
The quotient  is identified with the coordinate algebra of the $3$-dimensional $\theta$-sphere $S^3_\theta$, of which $S^4_\theta$ is a suspension. 

Condition \eqref{cond-ideal} holds, being $x$ central in the algebra $\mathcal{O}(S^7_\theta)$, and we can form 
the push-forward algebra $\mathcal{O}(S^3_\theta) \ot_{\mathcal{O}(S^4_\theta)}^\ta \mathcal{O}(S^7_\theta)$ along the  projection  $\pi: \mathcal{O}(S^4_\theta) \to  \mathcal{O}(S^3_\theta)$. The resulting algebra is an $\mathcal{O}(SU(2))$--Galois extension of $\mathcal{O}(S^3_\theta)$ (by Corollary~\ref{cor:quot}).
\begin{lem}
The $\mathcal{O}(SU(2))$--Galois extension of $\mathcal{O}(S^3_\theta) \subseteq \mathcal{O}(S^3_\theta) \ot_{\mathcal{O}(S^4_\theta)}^\ta \mathcal{O}(S^7_\theta)$ is trivial.
\end{lem}
\begin{proof}
By Lemma \ref{lem:IA}, there is the  isomorphism $\mathcal{O}(S^3_\theta) \ot_{\mathcal{O}(S^4_\theta)}^\ta \mathcal{O}(S^7_\theta) \simeq \mathcal{O}(S^7_\theta) \slash \langle x \rangle$ of $\mathcal{O}(SU(2))$--comodule algebras. 
In the quotient algebra of $\mathcal{O}(S^7_\theta)\slash \langle x \rangle$ one has
$z_1 z_1^* + z_2 z_2^* = z_3 z_3^* + z_4 z_4^* = \frac{1}{2}$ and the  $*$-subalgebra generated by $z_1, z_2$, which is commutative, is isomorphic to   $\mathcal{O}(SU(2))$.
We define a map $\gamma: \mathcal{O}(SU(2)) \to \mathcal{O}(S^7_\theta) \slash \langle x \rangle$ by setting  it  
on generators as  
$$ 
\begin{pmatrix}  
 w_1 & -w_2^*
\vspace{2pt}
\\
w_2  & w_1^* 
\end{pmatrix} 
\mapsto
\sqrt{2} \begin{pmatrix}  
z_1 & -z_2^*
\vspace{2pt}
\\
z_2  & z_1^* 
\end{pmatrix}
$$
and extending it as a morphism of (commutative) algebras.
The map $\gamma$ is an $H$-comodule morphism (cf. \eqref{princ-coactSU2}), and is convolution invertible with inverse  $\textsf{w} \mapsto 
 \sqrt{2} \small\begin{pmatrix}  
z_1^* & \! z_2^*
\\
-z_2  & \! z_1 
\end{pmatrix}$.  
The map $\gamma$ defines a cleaving map for the extension, proving that the extension is trivial.
\end{proof}

\medskip

\noindent
\textbf{The push-forward   to the quotient algebra  $\mathcal{O}(S^2)$}.
Let us  consider the quotient  of $\mathcal{O}(S^4_\theta)$ by the two-sided algebra ideal $I$ generated by $\alpha, \alpha^*$: 
$$
C= \mathcal{O}(S^4_\theta) \slash \langle \alpha, \alpha^* \rangle  \simeq \mathcal{O}(S^2).
$$
The resulting quotient algebra is the commutative algebra of coordinates of the $2$-sphere $S^2$.
 The set
 $ \big\{ Z_{m_j n_j}:=\prod_{j=1}^4 z_j^{m_j} (z_j^*)^{n_j} , \; m_j, n_j \in \mathbb{N}, \;   m_4 n_4=0 \big\}$ is a vector space basis of $\mathcal{O}(S^7_\theta)$.
 From the commutation relations in \eqref{comm7-4} it follows that, for each element of the basis, 
 $\alpha Z_{m_j n_j}$ is proportional to $\Z_{m_j n_j}\alpha$, and similarly for $\alpha^*$. Thus condition \eqref{cond-ideal} holds,
$\mathcal{O}(S^7_\theta) \cdot I \subseteq I \cdot \mathcal{O}(S^7_\theta)$, and we can form 
the push-forward algebra $\mathcal{O}(S^2) \ot_{\mathcal{O}(S^4_\theta)}^\ta \mathcal{O}(S^7_\theta)$ along the  projection  $\pi: \mathcal{O}(S^4_\theta) \to \mathcal{O}(S^2)$. Its  multiplication
is  as in \eqref{prod-quot} with
$\tap$ the linear map 
\beq\label{ta-S2}
\tap: \mathcal{O}(S^7_\theta) \ot \mathcal{O}(S^2) \to \mathcal{O}(S^2) \ot_{\mathcal{O}(S^4_\theta)} \mathcal{O}(S^7_\theta) \; , \quad Z_{m_j n_j} \ot [b] \mapsto 1 \ot_{\mathcal{O}(S^4_\theta)} Z_{m_j n_j} \, b \, . 
\eeq
We now simply write  $x, \beta, \beta^*$ for the algebra generators of the  quotient $\mathcal{O}(S^2)$, with   relation $\beta^* \beta  + x^2 =1$; a vector space basis of $\mathcal{O}(S^2)$ is  
$\{C_{r s t}=x^{r} \beta^s  (\beta^*)^t \, ,\; r,s,t \in \mathbb{N},\, st=0\}$. Then
\begin{align*}
\tap \big(Z_{m_j n_j} \ot C_{r s t} \big) &= 
1 \ot_{\mathcal{O}(S^4_\theta)}Z_{m_j n_j} \big( x^{r} \beta^s  (\beta^*)^t  \big) \\
& =1 \ot_{\mathcal{O}(S^4_\theta)}  \mu^\eta \big( x^{r} \beta^s  (\beta^*)^t  \big) Z_{m_j n_j} 
= \mu^\eta \, C_{r s t}  \ot_{\mathcal{O}(S^4_\theta)}Z_{m_j n_j} 
\end{align*}
 with proportionality coefficient $\mu^\eta$ whose exponent is computed out of \eqref{comm7-4}. 
 The resulting algebra $\mathcal{O}(S^2) \ot_{\mathcal{O}(S^4_\theta)}^\ta \mathcal{O}(S^7_\theta)$ is an
$\mathcal{O}(SU(2))$--Galois extension of $\mathcal{O}(S^2_\theta)$.
By Lemma \ref{lem:IA}, there is the  isomorphism $\mathcal{O}(S^2) \ot_{\mathcal{O}(S^4_\theta)}^\ta \mathcal{O}(S^7_\theta) \simeq \mathcal{O}(S^7_\theta) \slash \langle \alpha, \alpha^* \rangle$ of $\mathcal{O}(SU(2))$--comodule algebras.

 \section{Comparing Ehresmann--Schauenburg bialgebroids}\label{sec:ES}
In this section we consider the Ehresmann--Schauenburg bialgebroid of a Hopf--Galois extension obtained via the push-forward construction and compare it with that of the original 
Hopf--Galois extension. 

\subsection{Bialgebroids}
Let us  recall some definitions  and basic results on rings, corings and bialgebroids; we refer to \cite{BW} for more details.

Let $B$ be an algebra. A \textit{$B$-ring}, or an \textit{algebra over $B$}, is a pair $(L,\inc)$ where $L$ is an algebra (over $\kk$) and $\inc: B \to L$ is an algebra morphism, see e.g. \cite[\S 31.1]{BW}.  Then $L$ is a $B$ bimodule with $b \cdot x \cdot b'= \inc(b) x \inc(b')$, for $b,b' \in B$ and $x \in L$. A morphism between a $B$-ring $ (L,\inc)$ and a $C$-ring $(M, v)$, for $C$ an algebra, is a pair $(\ff , \gamma)$ where 
$\ff: B \to C$ and $\gamma: L \to M$
are algebra morphisms such that $ \gamma \circ \inc= v \circ \ff$.
Given an algebra morphism $\ff: B \to C$, any $C$-ring $(M, v)$ defines a $B$-ring $(\widetilde{M}, \tilde{v})$ with $\widetilde{M}=M$ and algebra map $\tilde{v}:=v \circ \ff: B \to M$.

\begin{ex}We are interested in \textit{algebras over  enveloping algebras}, see e.g. \cite[\S 31.2]{BW}.
An algebra $L$ with an algebra map $\inc : B \ot B^{op} \to L$ is a $(B\ot B^{op})$-ring
if and only if there exist algebra maps $\ss: B \to L$ and $\tt: B^{op} \to L$ with  
$\ss(b)\tt(b')= \tt(b')\ss(b)$, for all $b,b' \in B$. 
Indeed, given a $(B\ot B^{op})$-ring $(L, \inc)$   the restrictions maps
$$
\ss := \inc \circ ( \, \cdot \, \ot_B 1_B ) : B \to L \quad \mbox{and} \quad \tt := \inc \circ ( 1_B \ot_B  \, \cdot \, ) : B^{op} \to L
$$
are  algebra maps with the required property.
Vice versa, given $(L,\ss,\tt)$ one defines the algebra map $\inc: B \ot B^{op} \to L$, sending $b \ot b' \mapsto \ss(b)t(b')$.
The maps $\ss$ and $\tt$ are  respectively  called the  source and the  target map of the
$(B\ot B^{op})$-ring $L$.  
\\
Given an algebra morphism $\ff: B \to C$, any $C\ot C^{op}$-ring $(M, \ss, \tt)$ defines a $B \ot B^{op}$-ring $(\widetilde{M}, \tilde{\ss}, \tilde{\tt})$ with $\widetilde{M}=M$ and  source and target maps $\tilde{\ss}:= \ss \circ \ff$ and
$\tilde{\tt}:= \tt \circ \ff$.  
\end{ex}

Dually, see e.g.  \cite[\S 17.1]{BW}, with $B$ an algebra, a  \textit{$B$-coring} is a
triple $(L ,\underline{\Delta},\underline{\varepsilon})$ where $L $ is a $B$-bimodule and 
$\underline{\Delta}:L \to L \ot_{B} L $ and $\underline{\varepsilon}: L \to B$
are $B$-bimodule morphisms satisfying the 
coassociativity and counit conditions
\begin{align} 
(\underline{\Delta} \ot _{B} \id_L )\circ \underline{\Delta}= (\id_L  \ot _{B} \underline{\Delta})\circ \underline{\Delta}, \quad
(\underline{\varepsilon} \ot _{B} \id_L )\circ \underline{\Delta} = \id_L  =(\id_L  \ot _{B} \underline{\varepsilon})\circ \underline{\Delta}.
\end{align}
A morphism between a $B$-coring $L$ and a $C$-coring $M$ is a pair $(\gamma, \ff )$ where
$\ff: B \to C$ is an algebra morphism (also used to give to $C$ a $B$-bimodule structure) and  $\gamma: L \to M$ is a $B$-bimodule morphism   
such that
\beq\label{prop-morph-corings}
\xi \circ (\gamma \otb \gamma) \circ \underline{\Delta}_L =  \underline{\Delta}_M \circ \gamma \; , \qquad
\underline{\varepsilon}_M \circ \gamma = \ff \circ \underline{\varepsilon}_L
\eeq
where $\xi: M \ot_B M \to M \ot_C M$ 
 is the canonical morphism of $B$-bimodules induced by the map $\ff$, see \cite[\S 24.1]{BW}.
We write  $\underline{\Delta} (x)= \coone{x} \ot_{B}  \cotwo{x}$ for the comultiplication of any $x \in L$.

Given an algebra morphism $\ff: B \to C$, from any  $B$-coring $( L, \underline{\Delta}, \underline{ \varepsilon})$ one can form a $C$-coring $(\widetilde{L}, \underline{\widetilde{\Delta}}, \underline{\widetilde{\varepsilon}})$, the \textit{base ring extension of $L $}, \cite[\S 17.2]{BW}. This consists of the $C$-bimodule 
\beq\label{tildeL}
\widetilde{L}:=C \ot_B L \ot_B C
\eeq 
with  comultiplication 
\begin{align}\label{com-tildeL}
\underline{\widetilde{\Delta}}:& \; C \ot_B L \ot_B C \to C \ot_B L \ot_B C \ot_{C} C \ot_B L \ot_B C \simeq 
 C \ot_B L \ot_B C   \ot_B L \ot_B C \, , \nn \\
 & \; c \ot_B x \ot_B c' \; \mapsto (c \ot_B \coone{x}) \ot_{B} 1_C \ot_B  
 (\cotwo{x} \ot_B c')
 \end{align}
and counit 
\beq\label{cou-tildeL}
\underline{\widetilde{\varepsilon}}:  C \ot_B L \ot_B C \to C  \; , \quad c \ot_B x \ot_B c'  \mapsto 
 c \cdot \underline{\varepsilon}(x) \cdot c' = c \, \ff \left( \underline{\varepsilon}(x)  \right) c' .
\eeq

Then a morphism between a $B$-coring $L$ and a $C$-coring $M$ can  equivalently be seen  as the data of an algebra map $\ff: B \to C $  and  a morphism of $C$-corings  $\widetilde\gamma: \widetilde{L}  \to M$.

A \textit{$B$-bialgebroid} $(L , \ss, \tt, \underline{\Delta}, \underline{\varepsilon})$ is a  $(B\ot B^{op})$-ring  $(L , \ss, \tt)$ and a  
 $B$-coring   $(L  , \underline{\Delta}, \underline{\varepsilon})$ with compatibility among the structures:
 \begin{enumerate}[label=(\roman*)]
\item the $B$-bimodule structure of the $B$-coring   $L$ is given  via the source and the target maps:
$b \cdot x  = \ss(b)   x$, $x \cdot b=  \tt( b) x$   for all $b \in B$, $x\in L$; 

\vspace{.2 cm}
\item  
$Im (\underline{\Delta}) \subseteq 
L  \times_{B} L  := \left\{\ \sum\nolimits_j x_j\ot_{B} y_j\ |\ \sum\nolimits_j x_j \tt(b) \ot_{B} y_j =
\sum\nolimits_j x_j \ot_{B}  y_j \ss(b), \; \forall  \, b \in B \right\}
$
and the corestriction of $\underline{\Delta}: L \to L  \times_{B} L$ is an algebra map for 
$L  \times_{B} L$ with component-wise multiplication; 
\vspace{.2 cm}
\item
the counit $\underline{\varepsilon} : L  \to B$ is a unital map,   $\underline{\varepsilon} (1_{L })=1_{B}$,
and  satisfies
$\underline{\varepsilon}(x \, \ss(\underline{\varepsilon}(y)))=\underline{\varepsilon}(xy)=\underline{\varepsilon}(x \, t (\underline{\varepsilon}(y)))$,  for all $x,y \in L$.
\end{enumerate}

A morphism between a $B$-bialgebroid $L$ and a $C$-bialgebroid $M$  is a morphism $(\gamma, \ff )$
 between the $B$-coring $L$ and the $C$-coring $\M$ such that  $\gamma: L \to \M$ is also an algebra map.

\begin{ex}[\textit{The Ehresmann--Schauenburg bialgebroid}]
Let $B \subseteq A$ be a faithfully flat $H$-Galois extension. Denote by  $\tau$ the translation map.
Then the $B$-bimodule 
\beq\label{ES}
\coL  :=\{a\ot  a' \in A\ot  A \, | \, \zero{a} \ot  \tau(\one{a})a' = (a \ot  a') \ot _B 1_A\}
\eeq
is a bialgebroid,  called the  Ehresmann--Schauenburg bialgebroid (see  \cite[\S 34.14]{BW}).
It is a $B$-coring with comultiplication and counit
\beq\label{costr-L}
\underline{\Delta}(a \ot a' ) := \zero{a}\ot  \tau(\one{a})\ot a'  = 
(\zero{a} \ot \tone{\one{a}}) \ot_B (\ttwo{\one{a}} \ot a'), \qquad
\underline{\varepsilon}(a \ot  a') : = a a' ,
\eeq
for all $a \ot  a'  \in \coL$,  
and a $B \ot B^{op}$-ring with algebra multiplication
$$
(a \ot  a' ) \cdot (d \ot  d' ) :=  ad \ot d' a'  \, , \quad   \forall a \ot  a' , \,d \ot  d'  \in \coL ,
$$
thus $\coL$
 is a subalgebra of $A \ot A^{op}$.
 The  source and the target   maps are $\ss(b)=b\ot 1_A$ and $\tt(b)=1_A \ot  b$. 
The $B$-bimodule  $\coL$ coincides with the set  $(A\ot A)^{coH}$ of elements of $A \ot A$ that are coinvariants  for the diagonal $H$-coaction on $\delta^{A\ot  A}$, 
\beq\label{ESbis}
\coL  \simeq (A \ot A)^{coH} = \{ a\ot  a' \in A\ot  A \, | \, \zero{a}\ot  \zero{a'} \ot \one{a}\one{a'} = a\ot  a' \ot 1_H\}.
\eeq
This also implies that $aa'\in B$.
\end{ex}

\subsection{The Ehresmann--Schauenburg bialgebroid of a push-forward} 
Let $B \subseteq A$ be a faithfully flat $H$-Galois extension and let $\coL$ be its associated Ehresmann--Schauenburg bialgebroid. 
Let  $F: B \to C$ be a morphism of  algebras. Let $\ta:A\ot C\to C\ot A$ 
be a  twisting map which is a right $H$-comodule morphism and such that 
$\tap$ is a $B$-bimodule morphism which satisfies \eqref{nor-B} and  
$\tap(a \ot cc') = \taa{c}  \tap(\taa{a} \ot c')$.

The resulting twisted push-forward algebra $\Omega =: \pf$ of $B \subset A$ along the map $\ff:B \to C$ is a faithfully flat $H$-Galois extension of $C$, as from Theorem \ref{thm:pushHG}.
Its associated Ehresmann--Schauenburg bialgebroid $\coM$ is 
\begin{align}\label{ES-M}
 & \coM= \{x\ot  x' \in \Omega \ot \Omega   \, | \; \zero{x} \ot  \tau^\ta (\one{x})  \dta x' =x \ot  x' \ot _C (1_C \otb 1_A) \} \nn
\\
& \quad = \big\{ (c \otb a) \ot  (c' \otb a') \in \Omega \ot \Omega\, \big| \;  
\\
& (c \otb \zero{a})\hspace{-1pt} \ot \hspace{-1pt}(1_C \ot_B \tone{\one{a}}) \hspace{-1pt}\ot_C 
\hspace{-1pt}(\taa{c'}\hspace{-1pt} \otb \hspace{-1pt}\taa{ (\ttwo{\one{a}})} a') = (c \otb a) \hspace{-1pt}\ot \hspace{-1pt}(c' \otb a')\hspace{-1pt} \ot _C \hspace{-1pt}(1_C \otb 1_A) \big\}. \nn 
\end{align}

\noindent
The $C$-bimodule structure is given by left and right multiplication: using \eqref{nor-B}, for $c'' \in C$,
\begin{align} \label{coM-Cbim}
c''  \cdot \big( c \otb a \ot c' \otb a' \big) &= c'' c \otb a \ot  c' \otb a'  \; , \quad \nn \\
\big(c \otb a \ot c' \otb a' \big) \cdot c'' &= c \otb a \ot c' \taa{c''} \otb \taa{a'} .
\end{align}
As a $C$-bimodule, $\coM$ coincides with the set   $(\Omega \ot \Omega)^{coH}$ of elements   coinvariants  for the diagonal $H$-coaction (cf. \eqref{ESbis}). 
As from \eqref{costr-L} and \eqref{tau-pf}, $\coM$ has comultiplication
\beq\label{comult-M}
\underline{\Delta}_\coM \big( c \otb a \ot c' \otb a' \big) =  
\big(c \otb \zero{a} \ot 1_C \ot_B \tone{\one{a}} \big) \ot_C \big(1_C \ot_B \ttwo{\one{a}} \ot c' \otb a' \big)  . 
\eeq

\noindent For the counit we first observe that since $(c \otb a) \ot (c' \otb a') \in (\Omega \ot \Omega)^{coH}$,  
$$
(c \otb \zero{a}) \ot (c' \otb \zero{a'})  \ot \one{a}\one{a'} = (c \otb a) \ot (c' \otb a') \ot 1_H ,  
$$
and applying $m^\ta \ot \id_H$, we have 
$$
(c \taa{c'} \otb \taa{\zero{a}} \zero{ a'}) \ot \one{a} \one{ a'}  = (c  \taa{c'} \otb \taa{a}  a')\ot 1_H.
$$
Using that $\ta$ is an $H$-comodule morphism, the previous equation implies that $c \taa{c'} \otb \taa{a} a'$ belongs  to $C \otb B$: 
\begin{align*}
 c \taa{c'} \otb  \delta(\taa{a} a' )
 &= (c \taa{c'} \otb \zero{\taa{a}} \zero{a'}) \ot \one{\taa{a}} \one{ a'} 
= (c \taa{c'} \otb \taa{\zero{a}} \zero{ a'}) \ot \one{a} \one{ a'} 
\\
&= (c  \taa{c'} \otb \taa{a}  a')\ot 1_H.
\end{align*}
From the expression for the counit in \eqref{costr-L}
one then gets  
\begin{align}\label{counit-M}
\underline{\varepsilon}_\coM  \big((c \otb a) \ot  (c' \otb a' )\big)  
= c \taa{c'} \otb \taa{a} a' = c \taa{c'} \ff \left( \taa{a} a' \right) \otb 1_A \in C \otb 1_A \simeq C .
\end{align}

The algebra morphism $u_A  : A \to C \otb^\ta A$, $a \mapsto 1_C \otb a$ (see \eqref{projpi_alg}) induces a 
coring morphism between $\coL$ and $\coM$.  
\begin{lem}\label{lem:LM}
There exists a   morphism of bialgebroids
$(\gamma, \ff)$ between the $B$-bialgebroid $\coL$ and the  $C$-bialgebroid $\coM$,  given by the algebra map $\ff:B \to C$
and the $B$-bimodule map 
\beq\label{mappa-coring}
\gamma:  \coL \to \coM \, , \quad a \ot a' \mapsto (1_C \otb a) \ot (1_C \otb a') .
\eeq
\end{lem} 
\begin{proof}
The map $\gamma$ in \eqref{mappa-coring} is a $B$-bimodule map:
\begin{align*}
b(a \ot a') b' = ba \ot ab' & \stackrel{\gamma}{\longmapsto} (1_C \otb b a) \ot (1_C \otb a' b') \\
&= (\ff(b) \otb a) \ot 
(1_C \otb a' b') = b \big((1_C \otb a) \ot (1_C \otb a') \big)b'.
\end{align*}
We need to show that \eqref{prop-morph-corings} holds. 
The condition on the counit uses \eqref{counit-M}:
$$(\underline{\varepsilon}_\coM \circ \gamma) ( a \ot a')
= \underline{\varepsilon}_\coM \big((1_C \otb  a) \ot (1_C \otb a' )\big)
=   \ff(aa') \otb 1_C
= \ff \left( \underline{\varepsilon}_\coL ( a \ot a') \right)\otb 1_C .$$
For the comultiplication we compute 
\begin{align*} 
\xi \circ (\gamma \otb \gamma) \circ \underline{\Delta}_\coL  (a \ot a' ) 
&= \gamma \left(   \zero{a} \ot \tone{\one{a}}\right) \ot_C \gamma \left(\ttwo{\one{a}} \ot a' \right)
\\
& = (1_C \otb  \zero{a}) \ot (1_C \otb \tone{\one{a}}) \ot_C (1_C \otb \ttwo{\one{a}}) 
\ot (1_C \otb a') \\
&=\underline{\Delta}_\coM  \left( (1_C \otb a) \ot (1_C \otb a') \right) 
  \\
  &= 
 \underline{\Delta}_\coM  \left( \gamma  (a \ot a' ) \right) ,
 \end{align*}
 using \eqref{comult-M} for the last but one equality. 
 This shows that $\gamma$ in \eqref{mappa-coring} is a coring morphism.  It is also an algebra map since $\tap$ satisfies the condition \eqref{nor-B}. Thus $\gamma$ is a morphism of bialgebroids.
\end{proof}

Given the algebra map $\ff: B \to C$, this extends to an algebra map $  B \ot B^{op} \to  C \ot C^{op}$.  Out of the $C \ot C^{op}$-ring $\coM$   with source and target maps
$$
\ss: C \to \coM \,, \; c \mapsto (c \otb 1_A) \ot (1_C \otb 1_A) 
\quad \quad
\tt: C^{op} \to \coM \,, \; c \mapsto (1_C \otb 1_A)  \ot (c \otb 1_A) 
$$
we construct the $B \ot B^{op}$-ring $\widetilde{\coM}$. It coincides with $\coM$ as an algebra while it has  
  source and target maps
\begin{align*}
&\ss_{\widetilde{\coM}}: B \to \widetilde{\coM} \,, \quad b \mapsto (\ff(b) \otb 1_A) \ot (1_C \otb 1_A)  = (1_C \otb b) \ot (1_C \otb 1_A)
\\
&\tt_{\widetilde{\coM}}: B^{op} \to \widetilde{\coM} \,, \quad b \mapsto (1_C \otb 1_A)  \ot (\ff(b) \otb 1_A) =(1_C \otb 1_A)  \ot (1_C \otb b).
\end{align*}
Lemma \ref{lem:LM} leads to the following result.
\begin{prop}
The algebra map $\gamma$ in \eqref{mappa-coring}, seen as taking values in $\widetilde{\coM}$, 
$$
\gamma:  \coL \to \widetilde{\coM} \, , \quad (a \ot a') \mapsto (1_C \otb a) \ot (1_C \otb a') 
$$  
is a $B \ot B^{op}$-ring homomorphism, with $\ss_{\widetilde{\coM}}=\gamma \circ \ss_{\coL}$ and 
$\tt_{\widetilde{\coM}}=\gamma \circ \tt_{\coL}$. 

\end{prop}

 We next analyse the relationship between the $C$-corings $\coM$ and the base ring extension $\widetilde{\coL}$ of  $\coL$ obtained from the algebra map $\ff: B \to C$, see \eqref{tildeL}. The  $C$-coring $(\widetilde{\coL}, \underline{\widetilde{\Delta}}, \underline{\widetilde{\varepsilon}})$ is the $C$-bimodule 
\begin{align}\label{tildeL2}
\widetilde{\coL} & := C \otb \coL \otb C 
\\
& \:= \big\{ c \otb a \ot  a' \otb c' \in C \otb A\ot  A  \otb C\, | \; \zero{a} \ot  \tone{\one{a}} \otb \ttwo{\one{a}} a' =a \ot  a' \ot _B 1_A \big\}  \nn 
\end{align}
with comultiplication \eqref{com-tildeL} and counit \eqref{cou-tildeL}: 
\begin{align} \label{com-tildeL2}
& \hspace{-25pt}\underline{\Delta}_{\tilde\coL}:  \; 
c \otb a \ot a' \otb c' \; \mapsto 
\big(c \ot_B \zero{a} \ot  \tone{\one{a}} \otb  1_C \big) \ot_C \big(1_C \otb \ttwo{\one{a}} \ot a' \ot_B c'\big)
\\ \label{cou-tildeL2}
& \hspace{-25pt} \underline{ \varepsilon}_{\tilde\coL}:  \;   c \ot_B a \ot a' \ot_B c'  \mapsto 
 c \, \ff(aa') c' 
  \end{align}
(here $\ff(aa')$ is well-defined since $aa' \in B$ for $a \ot a' \in \coL$).

\begin{prop} 
Suppose the projection $\tap$ of the twisting map $\ta$ restricts to a well-defined map 
$\tap : A \otb C \to C \otb A$, that is \eqref{innerB} holds.
Then there exists a $C$-coring morphism  
$$
\widetilde{\gamma} = \id_{C \otb A} \ot \tap:  \widetilde\coL \to \coM , \quad
c \otb a \ot a' \otb c' \mapsto  c \otb a \ot \taa{c'} \otb \taa{a'} .
$$
Provided $\tap$ is invertible, the map $\widetilde{\gamma}$ is an isomorphism. 
\end{prop}

\begin{proof}
We   show that   $\widetilde{\gamma}$ takes values in $\coM$ that is, for all $c \otb a \ot  a' \otb c' \in \widetilde\coL$, 
\begin{align}\label{eq2}
(c \otb \zero{a}) \ot (1_C \ot_B \tone{\one{a}})  \ot_C 
(\taadue{\taa{c'}} & \otb \taadue{ (\ttwo{\one{a}})} \taa{a'})   
\\
&= (c \otb a) \ot (\taa{c'} \otb \taa{a'}) \ot _C (1_C \otb 1_A) .  \nn 
\end{align}
If $c \otb a \ot  a' \otb c' \in \widetilde\coL$, from the definition \eqref{tildeL2} of $\widetilde{L}$ one has \beq\label{eq1}
c \ot  \zero{a} \ot 1_C  \ot \tone{\one{a}} \otb \ttwo{\one{a}} a' \ot c' = c \ot a \ot 1_C \ot  a' \otb 1_A \ot c'.
\eeq
We apply $\id_{C \otb A} \otb \ta$ to the above equation, observing that it maps $(C \otb A) \otb (A \otb C)$ to 
$(C \otb A) \ot_C (C \otb A)$
(the map  $\id_{C \otb A} \ot \ta$ restricts to a well-defined map  on $C \otb A \otb A \otb C$ only when considered to take value in the quotient
$(C \otb A) \ot_C (C \otb A)$):
\begin{align*}
(c \otb  \zero{a}) \ot (1_C  \otb \tone{\one{a}}) \ot_C \! (\taa{c'} \otb \taa{(\ttwo{\one{a}} a')} ) 
&  = (c \otb a) \ot (1_C \otb  a') \ot_C (c' \otb 1_A) \\
& \hspace{-1cm} = (c \otb a) \ot (\taa{c'} \otb \taa{a'}) \ot _C (1_C \otb 1_A)
\end{align*}
using \eqref{cta=ac}  for the last equality. Then we obtain equality \eqref{eq2} once using the property 
$\tap(aa' \ot c) =  \taadue{\taa{c}} \otb \taadue{a}\taa{a'}$,
which follows from the analogous property $\tap(a \ot cc') = \taa{c}  \tap(\taa{a} \ot c')$ assumed on $\tap$
(see discussion in \S \ref{sec:twist} after equation \eqref{nor}).

We next show that  
$\underline{\Delta}_\coM \circ \widetilde{\gamma}=(\widetilde{\gamma} \ot_C \gamma)\circ \underline{\Delta}_{\tilde\coL}$ and
$\underline{\varepsilon}_\coM \circ \widetilde{\gamma}=\underline{\varepsilon}_{\tilde\coL}$.
From the expression of $\underline{\Delta}_{\tilde\coL}$ in \eqref{com-tildeL2} and of 
$\underline{\Delta}_\coM$ in \eqref{comult-M}, 
using $\tap(a \otb 1_C)=1_C \otb a$, for $a\in A$, we compute
 \begin{align*}
 ( \widetilde{\gamma} \ot_C \widetilde{\gamma})\big( \underline{\Delta}_{\tilde\coL} (c \otb a & \ot a' \otb c') \big) \\
 &= \big(c \ot_B \zero{a} \ot 1_C   \otb   \tone{\one{a}} \big) \ot_C \big(1_C \otb \ttwo{\one{a}} \ot \taa{c'} \ot_B \taa{a'}\big)
 \\
 &= \underline{\Delta}_\coM (c \otb a \ot  \taa{c'} \ot_B \taa{a'} ) 
 \\
 &= \underline{\Delta}_\coM \left(\widetilde{\gamma}(c \otb a \ot a' \otb c') \right). 
 \end{align*} 
 From the expression of 
$\underline{\varepsilon}_\coM$ in \eqref{counit-M},    we compute
 \begin{align*}
  \underline{\varepsilon}_{\coM} \big( \widetilde{\gamma} (c \otb a \ot a' \otb c') \big) 
 &= 
  \underline{\varepsilon}_{\coM} \big(  c \otb a \ot \taa{c'} \otb \taa{a'} \big)
 =   c \taadue{\taa{c'}} \ff \big( \taadue{a} \taa{a'} \big)  \ot_B 1_A
  \\
 &=   c  \taa{c'} \ff \big( \taa{(a a')} \big)   \ot_B 1_A = c \, \ff(aa') c'    \ot_B 1_A
 \end{align*} 
 where for the last but one equality we used the property  $\tap(aa' \ot c) =  \taadue{\taa{c}} \otb \taadue{a}\taa{a'}$
 for $a, a' \in A$ and $c\in C$ and for the last equality that
 $\tap(b \ot c) = \ff(b)c \otb 1_A$ for $c \in C$ and $b \in B$, and in particular for $b=aa' \in B$.  
 Then $\underline{\varepsilon}_{\coM} \big( \widetilde{\gamma} (c \otb a \ot a' \otb c') \big)$ coincides with  
 $ \underline{ \varepsilon}_{\tilde\coL} (c \ot_B a \ot a' \ot_B c' )$
from \eqref{cou-tildeL2} under the identification $C \otb 1_A \simeq C$.

When $\ta$ is invertible, then $\widetilde{\gamma}$ is an isomorphism.
\end{proof}

 \appendix 
 \section{On extensions of modules}\label{appA}
 
Given two rings $B$ and $\tb$ and a unital ring homomorphism $\ff :B \to \tb$,  
any right   (say) $\tb$-module $\mathcal{V}$ becomes a right $B$-module by setting $ v  \cdot b:=  v \cdot   \ff (b)$, for $b\in B$ and $v \in \mathcal{V}$.  

In particular $\tb$ itself may be treated as a right $B$-module.
Then, given a left $B$-module $\mathcal{E}$, its covariant $\ff$-extension is the left $\tb$-module
\beq\label{ff-ext}
\mathcal{E}_{(\ff)} :=  \tb \ot_{B}  \mathcal{E}  
\eeq
(see \cite[Ch. II \S 6]{CA56}). We write $\tilde{b} \ot_B \xi $ for (the class of) an element in $\mathcal{E}_{(\ff)}$.
If $\mathcal{E}$ is faithfully flat as left $B$-module, then  $\mathcal{E}_{(\ff)} =  \tb \ot_{B}  \mathcal{E}$  
is faithfully flat as left $\tb$-module (see \cite[Prop. 5 Ch.I \S 3]{bour}).

Moreover, from \cite[Ch. I Prop. 6.1]{CA56}, if the $B$-module $\mathcal{E}$ is  projective, then the $\tb$-module $\mathcal{E}_{(\ff)} $ is projective. 
Finally, if $\mathcal{E}$ is finitely generated with generating set $\{\xi_j\}$, then $\mathcal{E}_{(\ff)} $ is finitely generated with generating set $\{1_{\tb} \ot_B \xi_j  \}$. 
Thus,  if $\mathcal{E}$ is both finitely generated and projective, in short finite projective, 
so is its covariant extension. 

\begin{lem}\label{lem:proj}
Suppose the left $B$-module $\mathcal{E}$ is finite projective with $\mathcal{E} \simeq B^N p$, for an idempotent $p \in \M_N(B)$, and $N \in \N$. Then 
$\mathcal{E}_{(\ff)}\simeq  {\tb}^N \ff(p)$ for the idempotent
 $\ff(p) \in \M_N(\tb)$.
 \end{lem}
 \begin{proof}
The map $\ff: B \to \tb$ is extended componentwise to  $B^N$.
We define the map 
$\lambda:  \tb \ot  B^N   p \to  {\tb}^N \ff(p)$ by setting
  $1_{\tb} \ot \xi    \mapsto  \ff (\xi)\ff(p)$ and   requiring it to be a left $\tb$-module map. 
Being the ring map $\ff$ multiplicative we have 
$$
1_{\tb} \ot b \xi  - \ff(b) \ot \xi   \mapsto  \ff (b \xi )\ff(p) - \ff(b) \ff (\xi)  \ff(p)=0.
$$ 
Thus the map descends to a well-defined $\tb$-module map
$\lambda: \tb \ot_{B}  B^N   p \to  {\tb}^N \ff(p)$.

Conversely, given $\tilde{\xi}= (\tilde{\xi}_k) \in {\tb}^N  \ff(p)$, with then $\sum_{k=1}^N   \tilde{\xi}_k  \ff (p_{kj})= \tilde{\xi}_j$, we associate to it  the element
$\lambda^{-1}(\tilde{\xi}):=\sum_{k=1}^N  \tilde{\xi}_k \ot_B p_{kj}  \in   \tb \ot_B B^N p$.
We have 
$$
\lambda (\sum_{k=1}^N \tilde{\xi}_k \ot_B p_{kj} )
= \sum_{k=1}^N  \tilde{\xi}_k \ff (p_{kj})  \ff(p)
= \tilde{\xi}\ff(p)   =\tilde{\xi}
$$
and for $\xi =(\xi_k) \in B^N p$, we have 
$$
\lambda^{-1} (\ff (\xi) \ff(p) )
=\sum_{k,l=1}^N \ff (\xi_l) \ff(p_{l k})  \ot_B p_{k j} 
=\sum_{k,l=1}^N    1_{\tb} \ot_B  \xi_l  p_{l k} p_{kj} 
=1_{\tb} \ot_B \xi  .
$$
Thus, $\mathcal{E}_{(\ff)}= \tb \ot_{B}  \mathcal{E}  \simeq  {\tb}^N \ff(p)$ as claimed.
 \end{proof}
  
 In the following we consider $B$ and $\tb$ to be unital $*$-algebras. Suppose $\mathcal{E}$ is a left $B$-module equipped with a $B$-valued inner product $\inner{\cdot}{\cdot}_B$ and with a standard module frame $\{ \eta_j \}_{j=1, \dots , N}$. That is (see \cite[Defin. 7.1]{Rff08}) there exists a finite family
 $\{ \eta_j \}$ of elements of $\mathcal{E}$ such that for any $\xi \in \mathcal{E}$  one has the reconstruction formula 
 \beq
 \xi = \sum_{j=1}^N  \inner{\xi}{\eta_j}_B \eta_j .
 \eeq
 Then, the $B$-module $\mathcal{E}$ is finite projective with $\mathcal{E} \simeq  B^N p$, where $p \in M_N(B)$ is the projection of components $p_{jk}:=\inner{\eta_j}{\eta_k}_B$.  
 Conversely (see \cite[\S 7]{Rff08}), any finite projective $B$-module $\mathcal{E}$ has a standard module frame.  If  $\mathcal{E} \simeq B^N p$ for some $p \in M_N(B)$, the restriction to  $\mathcal{E}$ of the inner product on $B^N$, $\inner{(b_j)}{(b'_j)}_B = \sum b_j^* b'_j$, is an inner product on $\mathcal{E}$ and $\{\eta_j:= \sum_k e_k p_{kj}\}$ is a standard module frame for $\{e_j\}$ the natural basis of the free module $B^N$.

 \begin{lem}\label{lem:smf}
Suppose $\mathcal{E}$ is a  left  $B$-module equipped with a $B$-valued inner product $\inner{\cdot}{\cdot}_B$ and with a standard module frame $\{ \eta_j \}_{j=1, \dots , N}$. Let $\ff: B \to \tb$ be a $*$-algebra map. Then 
\beq\label{in-pr-ex}
\inner{\tilde{b} \ot \xi }{\tilde{c} \ot \eta }_{\tb} := \tilde{b} \, \ff \big( \inner{\xi  }{\eta  }_{B}\big)\,  \tilde{c}^* \, , \quad 
\xi, \eta \in \mathcal{E},\; \tilde{b},\tilde{c} \in \tb
\eeq
is an inner product on $\mathcal{E} \ot \tb$ which
descends to a well-defined inner product on the extended $\tb$-module
$\mathcal{E}_{(\ff)}$  and $\{1_{\tb} \ot_B \eta_j  \}_{j=1, \dots , N}$ is a standard module frame for $\mathcal{E}_{(\ff)}$. 
 \end{lem}
 \begin{proof}
 To prove that \eqref{in-pr-ex} satisfies the properties of an inner product one uses  the analogous properties of $\inner{\cdot}{\cdot}_B$ and the fact that $\ff$ is a $*$-map.
 We show that \eqref{in-pr-ex} descends to a well-defined map on  $\mathcal{E}_{(\ff)}$:
 for $\xi, \eta \in \mathcal{E}$, $\tilde{b},\tilde{c} \in \tb$ and $b,c \in B$, one has:
\begin{align*}
 \inner{ \tilde{b} \ot  b \xi }{  \tilde{c}\ot  c \eta }_{\tb} 
 &= \tilde{b} \, \ff \big( \inner{b \xi  }{c\eta  }_{B}\big) \tilde{c}^* 
  = \tilde{b} \ff \big( b \inner{\xi}{\eta }_{B} c^* \big) \tilde{c}^* 
 = \tilde{b} \,\ff ( b) \ff \big(  \inner{\xi}{\eta }_{B}   \big)  \ff  (c^*) \tilde{c}^* 
 \\
 & =\tilde{b} \, \ff ( b)  \ff \big(  \inner{\xi}{\eta }_{B}   \big)  (\tilde{c}\, \ff  (c))^* 
 =  \inner{\tilde{b} \, \ff ( b) \ot  \xi    }{\tilde{c} \, \ff(c) \ot  \eta  }_{\tb} .
\end{align*}
 For each $\tilde{b} \ot_B \xi   \in \tb \ot_B \mathcal{E}  $,  
\begin{align*}
\tilde{b} \ot_B \xi
 &= \sum_{j=1}^N \tilde{b} \ot_B  \inner{\xi}{\eta_j}_B \eta_j
 = \sum_{j=1}^N \tilde{b}  \, \ff \big( \inner{\xi}{\eta_j}_B \big)  \ot_B  \eta_j 
 = \sum_{j=1}^N \tilde{b}  \, \ff \big( \inner{\xi}{\eta_j}_B \big) \left(  1_{\tb} \ot_B \eta_j \right) 
 \\
 & = \sum_{j=1}^N  \inner{\tilde{b}  \ot_B \xi }{1_{\tb} \ot_B \eta_j }_{\tb} \left(1_{\tb} \ot_B \eta_j \right) 
\end{align*}
showing that $\{ 1_{\tb} \ot_B \eta_j\}_{j=1, \dots , N}$ is a standard module frame for $\mathcal{E}_{(\ff)}$. 
 \end{proof}
 In agreement with Lemma \ref{lem:proj} we have the following.
 \begin{cor}
 In the hypothesis of Lemma \ref{lem:smf}, the module  $\mathcal{E}_{(\ff)}$ is projective with 
$\mathcal{E}_{(\ff)} \simeq  {\tb}^N \tilde{p}$, where $\tilde{p} \in M_N(\tb)$ is the projection of components  
$$
\tilde{p}_{jk}:=\inner{1_{\tb} \ot_B \eta_j}{1_{\tb} \ot_B \eta_k}_{\tb} = \ff \big( \inner{\eta_j}{\eta_k}_B \big).
$$
 \end{cor}

\section{On strong connections and $2$-cocycle deformations}\label{app:B}

This  Appendix is devoted to show that the property \eqref{prop-ell} of the strong connection of an $H$-Galois extension is preserved when the Hopf--Galois extension is deformed via a $2$-cocycle on the Hopf algebra $H$. 

From the general theory (due to Doi and Drinfeld), given a  $2$-cocycle $\cot: H \ot H \to \kk$ on a Hopf algebra $H$, the  latter can be deformed to a new Hopf algebra $\hg$.  This  
 consists of the coalgebra $H$ with associative product 
\beq\label{hopf-twist}
m_{\cot} (h \ot k):= h \mt k:= \co{\one{h}}{\one{k}} \,\two{h}\two{k}\, \coin{\three{h}}{\three{k}} \, , \quad h,k\in H \, 
\eeq
and   antipode  $S_\cot:= u_\cot *S *\bar{u}_\cot$.  Here
\beq\label{uxS}
u_\cot:  H\longrightarrow \kk \, , \; h\longmapsto \co{\one{h}}{S(\two{h})}  \, ,\quad 
\bar{u}_\cot: H\longrightarrow \kk \,, \; h \longmapsto \coin{S(\one{h})}{\two{h}} 
\eeq
are convolution inverse of each other. 
Moreover each (say) right $H$-comodule algebra $(A, \delta)$, with  coaction $\delta: A \to A \ot H$, is deformed to an 
$\hg$-comodule algebra $(A_\cot, \delta_\cot)$ given by the $\kk$-module $A$ with unchanged unit $\eta_\cot := \eta$  and deformed product  
 \beq\label{rmod-twist} 
 m_\cot : A_\cot \ot A_\cot \longrightarrow A_\cot \,, \;a\ot  a' \longmapsto \zero{a} \zero{a'} \,\coin{\one{a}}{\one{a'}} =: a \mtco a' \,.
 \eeq
 The coaction $\delta_\cot := \delta$ is unchanged, as a linear map, and $A^{coH}=A_\cot^{co \hg}$.
 
A categorical approach to $2$-cocycle deformations of Hopf--Galois extensions was developed in \cite{ppca}; we refer to it for details.
As shown in 
  \cite[Corollary 3.7]{ppca},   the extension $B:=A^{coH} \subseteq A$ is $H$-Galois if and only if $B \subseteq A_\cot$ is $H_\gamma$-Galois.   Moreover $B \subseteq A$ is faithfully flat if and only if $B \subseteq A_\cot$ is such.
This was proved in \cite[Corollary 3.9]{ppca} by exhibiting a strong connection for the deformed Hopf--Galois extension, given as a splitting $s_\cot: A_\cot \to B \ot A_\cot$ of the multiplication map. The expression for the corresponding map $\ell_\cot$ associated with $s_\cot$ was given in \cite{JZ}: for $\ell: H \to A \ot A$ the strong connection of the original $H$-Galois extension, 
\beq\label{ell-cot}
\ell_\cot: \hg \to A_\cot \ot A_\cot \, , \quad h \mapsto u_\cot(\one{h}) \ell(\two{h}).
\eeq

\begin{lem}
Let $B \subset A$ be a faithfully flat $H$-Galois extension admitting a strong connection $\ell$ with property \eqref{prop-ell}.
Then the strong connection $\ell_\cot$  of the deformed extension in \eqref{ell-cot} has the same property. 
\end{lem}
\begin{proof}
We use the notation  $\ell_\cot(h)=  {h}^{\scriptscriptstyle{<1>_\cot}} \ot  {h}^{\scriptscriptstyle{<2>_\cot}}$ for the image of an element $h \in \hg$. We have to show that $\ell_\cot$ satisfies \eqref{prop-ell}:
$$
\ell_\cot(h \mt k) = {k}^{\scriptscriptstyle{<1>_\cot}} \mtco {h}^{\scriptscriptstyle{<1>_\cot}}\ot  {h}^{\scriptscriptstyle{<2>_\cot}}\mtco  {k}^{\scriptscriptstyle{<2>_\cot}}.
$$
To compute the left hand side we need the property
\begin{align}\label{u-prod} 
u_\cot(hk) &= \co{\one{h} \one{k}}{S(\two{k}) S(\two{h}) } 
\nn  \\
&= \coin{\one{h}}{\one{k}} \, u_\cot(\two{k}) \, \coin{S(\three{k})}{S(\four{h})} \,
\co{\two{h}}{S(\three{h})} 
\nn  \\
& =  \coin{\one{h}}{\one{k}} \, u_\cot(\two{k}) \, u_\cot(\two{h}) \, \coin{S(\three{k})}{S(\three{h})} 
\end{align}
of the linear map $u_\cot$, see \cite[Lemma 3.2]{ppca}.  Then
\begin{align*}
\ell_\cot(h \mt k) &=   \co{\one{h}}{\one{k}} \, \ell_\cot(\two{h}\two{k})\, \coin{\three{h}}{\three{k}}
\\
&= \co{\one{h}}{\one{k}} \, u_\cot(\two{h}\two{k})\, \ell(\three{h}\three{k}) \coin{\four{h}}{\four{k}}
\\
&= \co{\one{h}}{\one{k}} \, u_\cot(\two{h}\two{k})\,  \lone{\three{k}} \lone{\three{h}}\ot \ltwo{\three{h}} \ltwo{\three{k}}  \coin{\four{h}}{\four{k}}
\\
&= u_\cot (\one{k}) u_\cot (\one{h})  \, \coin{S(\two{k})}{S(\two{h}) } 
\lone{\three{k}} \lone{\three{h}}\ot \ltwo{\three{h}} \ltwo{\three{k}}  \coin{\four{h}}{\four{k}}
\end{align*}
where we used property \eqref{prop-ell} of $\ell$ for the last but one equality and \eqref{u-prod} for the last one. 
For the right hand side we compute
\begin{align*}
&{k}^{\scriptscriptstyle{<1>_\cot}} \mtco {h}^{\scriptscriptstyle{<1>_\cot}}\ot  {h}^{\scriptscriptstyle{<2>_\cot}}\mtco  {k}^{\scriptscriptstyle{<2>_\cot}} 
\\
&= u_\cot (\one{k}) u_\cot (\one{h})  \, \lone{\two{k}} \mtco \lone{\two{h}}\ot \ltwo{\two{h}} \mtco \ltwo{\two{k}} 
\\
&= u_\cot (\one{k}) u_\cot (\one{h})  \, \zero{\lone{\two{k}} } \zero{\lone{\two{h}}} \,\coin{\one{\lone{\two{k}} }}{\one{\lone{\two{h}}}}
\ot \ltwo{\two{h}} \mtco \ltwo{\two{k}} 
\\
&= u_\cot (\one{k}) u_\cot (\one{h})  \,  \lone{\three{k}} \lone{\three{h}}   \,\coin{S(\two{k})}{S(\two{h}) }\ot \ltwo{\three{h}} \mtco \ltwo{\three{k}} 
\\
& = u_\cot (\one{k}) u_\cot (\one{h})  \,  \lone{\three{k}} \lone{\three{h}}   \,\coin{S(\two{k})}{S(\two{h}) }\ot 
\zero{\ltwo{\three{h}}} \zero{\ltwo{\three{k}} } \,\coin{\one{\ltwo{\three{h}}}}{\one{\ltwo{\three{k}} }} 
\\
& = u_\cot (\one{k}) u_\cot (\one{h})  \,  \lone{\three{k}} \lone{\three{h}}   \,\coin{S(\two{k})}{S(\two{h}) }\ot 
 \ltwo{\three{h}}  \ltwo{\three{k}}  \,\coin{ \four{h}}{\four{k}} 
\end{align*}
where we used property \eqref{prop-ell-S} for the third equality and \eqref{prop-ell0} for the last one.
Then, $\ell_\cot$ satisfies \eqref{prop-ell} as claimed.
\end{proof}

\newpage

\noindent
\textbf{Acknowledgements.}
GL acknowledges support from PNRR MUR projects PE0000023-NQSTI.
CP gratefully acknowledges support by the University of Naples Federico II under the grant FRA 2022 \emph{GALAQ: Geometric and ALgebraic Aspects of Quantization}.
Both authors are members of INdAM-GNSAGA.
We are grateful to Tomasz Brzezi\'nski, Francesco D'Andrea, Ulrich Kramer and Thomas Weber for useful discussions and suggestions.

\end{document}